\documentclass{birkjour}
\usepackage{amssymb}
\usepackage[mathlines,switch]{lineno}
\usepackage{color}
\usepackage{amsmath}
\usepackage{graphicx}

\allowdisplaybreaks
\renewcommand{\subjclassname}{Mathematics Subject Classification (2020)}

\makeatletter
\def\bjour@firstoffive#1#2#3#4#5{#1}
\def\bjour@secondoffive#1#2#3#4#5{#2}
\def\@newref#1{\expandafter\@setref\csname r@#1\endcsname\bjour@firstoffive{#1}}
\def\@newpageref#1{\expandafter\@setref\csname r@#1\endcsname\bjour@secondoffive{#1}}
\makeatother

\theoremstyle{plain}
\newtheorem{theorem}{Theorem}[section]
\newtheorem{lem}[theorem]{Lemma}
\newtheorem{prop}[theorem]{Proposition}
\newtheorem{cor}[theorem]{Corollary}
\newtheorem{quest}[theorem]{Question}
\theoremstyle{definition}
\newtheorem{ex}[theorem]{Example}
\newtheorem{rem}[theorem]{Remark}
\newtheorem{example}[theorem]{Example}
\newtheorem{remark}[theorem]{Remark}
\newtheorem{proposition}[theorem]{Proposition}
\numberwithin{equation}{section}

\renewcommand{\theenumi}{\roman{enumi}}
\renewcommand{\labelenumi}{(\theenumi)}

\newcommand{\cB}{\mathcal B}
\newcommand{\cC}{\mathcal C}
\newcommand{\cD}{\mathcal D}
\newcommand{\cF}{\mathcal F}
\newcommand{\cP}{\mathcal P}

\newcommand{\NN}{\mathbb N}
\newcommand{\RR}{\mathbb R}

\DeclareMathOperator{\Card}{card}
\DeclareMathOperator{\Rank}{rank}
\DeclareMathOperator{\supp}{supp}

\usepackage[hidelinks]{hyperref}

\begin{document}

%
%
%
%
%
%
%
%
%

\title[A core variety approach to the pure $Y=X^d$ TMP]{A core variety approach to the pure $Y=X^d$ truncated moment problem: Part 1}
\author[L. Fialkow]{Lawrence Fialkow}
\address{State University of New York, New Paltz, NY 12561, U.S.A.}
\email{fialkowl@newpaltz.edu}

\author[A. Zalar]{Alja\v{z} Zalar}
\address{University of Ljubljana, Faculty of Computer and Information Science \& \\
Faculty of Mathematics and Physics,\\
\& Institute of Mathematics, Physics and Mechanics, Ljubljana, Slovenia}
\email{aljaz.zalar@fri.uni-lj.si}

\thanks{
The second-named author was supported by the Slovenian Research Agency
program P1-0288 and grants J1-50002, J1-60011, J1-70017.}

\subjclass{Primary 44A60, 47A57, 47A20; Secondary 47N40.}

\keywords{Truncated moment problems, representing measure, core variety of a multisequence, moment matrix}
\date{\today}

\begin{abstract}
Let $\beta \equiv\beta^{(2n)}$ be a real bivariate sequence of degree $2n$.
We study the existence of  representing measures
for $\beta$ supported in the curve $y=x^{d}$ ($d\ge 1$)
in the case when all column dependence relations in the moment matrix
$M_n(\beta)$ are generated by the relation $Y=X^{d}$. We prove that the core variety
of $\beta$, $\mathcal{CV}(L_{\beta})$, is nonempty (equivalently, representing measures exist) if and
only if $C$, the partially defined \textit{core matrix} of $\beta$, admits a
positive, recursively generated completion $C[A]$.
Moreover,
 $\mathcal{CV}(L_{\beta})$
is the entire curve $y=x^{d}$
if and only if there is a positive definite completion $C[A]$.
In the remaining case, if there is a measure, it is unique and finitely
atomic.
For $d = 3$, we use these results to compute the core variety of $\beta$ and
give new characterizations of the existence of representing measures,
which complement a result of the first author.
\end{abstract}
\maketitle

\section{Introduction}

Given a bivariate sequence of degree $2n$,
\begin{equation}
\label{def:beta}
\beta\equiv \beta^{(2n)} = \{\beta_{ij}: i,j \ge 0, i+j \le 2n
       \},~\beta_{00}=1,
\end{equation}
and a closed set $K\subseteq \mathbb{R}^{2}$,
the Truncated $K$-Moment Problem (TKMP) seeks conditions
on $\beta$ such  that there exists a positive Borel
measure $\mu$ on $\mathbb{R}^{2}$, with
$\supp\mu \subseteq K$, satisfying
\[ \beta_{ij} =\int_{\mathbb{R}^{2}} x^{i}y^{j} d\mu(x,y) ~~~~(i,j\ge 0, i+j \le 2n);\]
 $\mu$ is a
 \textit{$K$-representing measure} for $\beta$.

A comprehensive reference for all aspects of the Moment Problem is the recent treatise of K.  Schm\"{u}dgen \cite{sch}.
Apart from solutions based on semidefinite programming and optimization,
 several different \textit{abstract} solutions to TKMP appear in the literature,
including the Flat Extension Theorem \cite{cf2005}, the Truncated Riesz-Haviland Theorem \cite{cf2008}, the idempotent approach of \cite{vas}, and, more recently,
the Core Variety Theorem \cite{bf}.
By a \textit{concrete} solution to TKMP we mean an implementation of one of the abstract
theories involving only basic linear algebra and solving algebraic equations (or estimating the size of the solution set).
The ease with which any of the abstract results can be applied
to solve particular moment problems in concrete terms varies considerably depending on the
problem, with most concrete results attributable to the Flat Extension Theorem and very few to the other approaches.
In the sequel we show how the Core Variety Theorem (Theorem \ref{cvt} below) can indeed be applied to certain concrete moment problems,  namely when $K$ is the planar curve $y =x^d$ ($d \geq 1$).

In the classical literature
TKMP has been solved concretely in terms of positive Hankel matrices when $K$ is the real line, the
 half-line $[0,+\infty)$, or the
closed interval $[a,b]$ (cf.\ \cite{ST,cf1991}). For the case when $K$ is a planar curve $p(x,y) = 0$ with $\deg p \le 2$,
TKMP has been solved concretely in terms of moment matrix extensions
(see Theorem \ref{degree2} below,
\cite{cf2002,cf2004,cf2005a,f2}).  In \cite{f1}  moment matrix extensions
are used to concretely solve the truncated moment problem for $y=x^{3}$ and to
solve (in a less concrete sense)
 truncated
moment problems on curves of the form $ y = g(x)$ and $yg(x) = 1$ ($g \in \mathbb{R}[x]$).

 More recently, several authors have intensively studied TKMP on certain planar curves of higher degree, using moment matrix extensions
and a ``reduction of degree" technique to improve and extend the results of  \cite{f1}
(cf.\ \cite{z1,z2,z3,z4,yz}).
We also note that for closed planar sets $K$ that are merely semi-algebraic, such as the closed
unit disk, very little is known concerning concrete solutions to TKMP (cf.\ \cite{cf2000}).
In a different direction, a recent concrete solution of the multidimensional TKMP on the vertices of the hypercube is given in \cite{ekt25}.

The problem of constructing positive cubature rules is related to, but distinct from, the TKMP. In the cubature problem, one starts with a positive Borel measure \(\mu\) supported on \(K\) and having finite moments through degree \(2n\). Thus, \(\mu\) is already a \(K\)-representing measure for its truncated moment sequence \(\beta\), and the problem is to construct a finitely atomic \(K\)-representing measure for \(\beta\) with as few atoms as possible, or to estimate the minimum number of atoms required. In particular, the existence of a representing measure is assumed in the cubature problem, whereas it is the principal issue in the TKMP. The smallest integer that bounds the number of atoms required for every moment functional in the corresponding moment cone is its Carath\'eodory number. Its finiteness follows from Richter's theorem \cite{Ric}, which guarantees a finitely atomic representing measure once the existence of a representing measure is known. General bounds for Carath\'eodory numbers on algebraic varieties and plane algebraic curves have been obtained using Hilbert-function and optimization methods \cite{diokummer,rs,RT26}. More recently, sharp Carath\'eodory numbers and related flat-extension results have been established for certain plane cubic curves with smooth projective closure \cite{bbs}. Complementing these results, \cite{kz} develops constructive solutions to the TKMP on plane cubic curves and obtains additional upper bounds for their Carath\'eodory numbers, some of which are sharp or differ from the sharp bound by at most one.

The results cited just above do not provide concrete solutions to TKMP for planar curves of the form $y=x^d$ $(d\geq 4)$. The aim of this note
 is to illustrate how the core variety, described in Theorem \ref{cvt}, can be used to study
TKMP for $K = \Gamma$,
  the  planar curve $y=x^{d}$ ($d\ge 1$),
  when the associated moment matrix $M_n(\beta)$ is \textit{$(y-x^d)$-pure}, i.e., the column dependence relations in $M_n(\beta)$ are precisely those that can be derived from the column relation $Y=X^d$ by \textit{recursiveness} and linearity (see just below for terminology and notation).
  The core variety of $\beta$ coincides with the
 union of supports of all representing
 measures for $\beta$, and in Section \ref{Section 2} we develop a core variety framework for studying TKMP in the $(y - x^{d})$-pure case.
 In Theorem \ref{Jmeasure} we prove that $\beta$
          has a representing measure if and only if $C$, the partially defined
\textit{core matrix} for $\beta$, admits a positive semidefinite, recursively generated completion $C[A]$.
The core variety of $\beta$  coincides with the entire curve $y=x^{d}$ if and only if there exists positive definite completion $C[A]$.
In the remaining case of a measure, it is unique, with support a finite subset of $\Gamma$.
In Section \ref{Section 4} we apply the results of Section \ref{Section 2} to
compute the core variety of $\beta$ in the $(y-x^{3})$-pure truncated moment problem (see Theorem \ref{y=x3}); this result subsumes
a result of \cite{f1} which used a lengthy flat extension construction to give a necessary and sufficient
condition for
 the existence of a representing measure.

\section{Preliminaries}
\label{sec:prel}

Although our focus in the sequel is TKMP for
 the planar curves $y = x^d$, we note that the following discussion, and the results we cite from
 \cite{b,bf,cf2005,cf2008,f3},
generalize to the \textit{multivariable} truncated moment problem.

Let $\mathcal{P} := \mathbb{R}[x,y]$ and
let $\mathcal{P}_{k} := \{ q \in \mathcal{P}\colon \deg q\le k\}$.
Given $\beta \equiv \beta^{(2n)}$ as in \eqref{def:beta}, define the
\textit{Riesz functional}
$ L_{\beta}:\mathcal{P}_{2n}\longrightarrow \mathbb{R}$ by
 $$
  \sum a_{ij}x^{i}y^{j} \longmapsto
  \sum a_{ij}\beta_{ij}.      $$
For a sequence $\beta \equiv \beta^{(2n)}$ with Riesz functional $L_{\beta}$, the
\textit{moment matrix} $M_{n}$ has rows and columns
indexed by the monomials in $\mathcal{P}_{n}$
in degree-lexicographic order, i.e.,
$\textit{1}, X, Y, X^{2}, XY, Y^{2},\ldots, X^{n}, \ldots, Y^{n}.$
In this case, the element of $M_{n}$ in row $X^{i}Y^{j}$, column $X^{k}Y^{l}$ is
$\beta_{i+k,j+l}$.
More generally,
for $r,s \in \mathcal{P}_{n}$, with coefficient vectors $\widehat{r},
\widehat{s}$ relative to the basis of monomials, we have
\begin{equation}
\label{def:evaluation-v2}
\langle M_{n} \widehat{r},\widehat{s}\rangle
:= L_{\beta}(rs).
\end{equation}
In the sequel, for $q \in \mathcal{P}_{n}$,
$q = \sum a_{ij}x^{i}y^{j}$,
we set
\begin{equation}
\label{def:evaluation}
    q(X,Y) := \sum a_{ij}X^{i}Y^{j}~ (=M_{n} \widehat{q}).
\end{equation}

If $\beta$ has a $K$-representing measure $\mu$, then
 $L_{\beta}$ is \textit{$K$-positive}, i.e.,
$q\in \mathcal{P}_{2n}, ~ q|K\ge 0 \Longrightarrow L_{\beta}(q)\ge 0$ (since $L_{\beta}(q)=\int_K q d\mu$).
 The converse is not true; instead, the Truncated Riesz-Haviland Theorem \cite{cf2008} shows
that $\beta$ admits a $K$-representing measure if and only if $L_{\beta}$
admits an extension to a $K$-positive linear functional on $\mathcal{P}_{2n+2}$.
In \cite{b} G. Blekherman proved that if $M_n$ is positive semidefinite and $\Rank M_{n}\le 3n-3$, then $L_{\beta}$ is $\mathbb{R}^{2}$-positive,
so the Truncated Riesz-Haviland Theorem then implies that $\beta^{(2n-1)}$ has a representing measure.
Using special features of the proof of Theorem \ref{f1thm} (below), in \cite{ef} C. Easwaran and the first-named author
exhibited a class of Riesz functionals that are positive but have no representing measure.
Apart from these results, it seems
 very difficult to verify positivity of Riesz functionals in examples without first proving the
existence of representing measures.

 Several basic \textit{necessary} conditions for a representing measure $\mu$ can be expressed in terms related to moment matrices
(cf. \cite{cf2005});
we will refer to these without further reference in the sequel:
\begin{enumerate}
\item 
$M_{n}(\beta)$ is \textit{positive semidefinite}: 
$$\langle M_{n} \widehat{r},\widehat{r}\rangle
= L_{\beta}(r^{2}) = \int r^{2} d\mu \ge 0~ (\forall r \in \mathcal{P}_{n}).$$
\item For any representing measure $\mu$, $\Card(\supp\mu) \ge \Rank M_{n}$.
\item Note that a dependence relation in the column space of $M_{n}$ can be expressed as
$r(X,Y)    = 0$, where $r\in \mathcal{P}_{n}$.
Define the \textit{variety} of $M_{n}$, $\mathcal{V}(M_{n})$,
 as the common zeros of the polynomials $r\in \mathcal{P}_{n}$
such that  $r(X,Y)    = 0$.
Then $\supp\mu \subseteq \mathcal{V}(M_{n})$, so
$\Card \mathcal{V}(M_{n}) \ge \Rank M_{n}$.
\item $M_{n}$ is \textit{recursively generated}: whenever $r$, $s$, and $rs$ are in $\mathcal{P}_{n}$
and $r(X,Y)    = 0$, then $(rs)(X,Y)    = {0}$.
\item $M_{n}$ (or $L_{\beta}$) is \textit{consistent}: for $p\in \mathcal{P}_{2n}$,
$p|\mathcal{V}(M_{n}) \equiv 0 \Longrightarrow L_{\beta}(p) = 0$;
 consistency implies recursiveness \cite{cfm}.
\end{enumerate}

The Flat Extension Theorem \cite{cf2005} shows that $\beta$ admits a representing measure if and only if
$M_{n}$ admits a positive semidefinite moment matrix extension $M_{n+k}$ (for some $k\ge 0$)
for which there is a rank-preserving (i.e., \textit{flat}) moment matrix extension $M_{n+k+1}$.
Using this result, in a series of papers R. Curto and the first-named author solved TKMP for
planar curves of degrees 1 and 2 as follows.

\begin{theorem}
[{\cite[Degree-2 Theorem]{cf2002,cf2004,cf2005a,f2}}]
\label{degree2}
Suppose $r(x,y)\in \mathcal{P}$ with $\deg r\le 2$.
For $n\ge \deg r$, $M_{n}$ has a representing measure supported in the curve $r(x,y) = 0$
if and only if $r(X,Y)    = 0$ and $M_{n}$ is positive semidefinite,
recursively generated, and satisfies $\Card \mathcal{V}(M_{n})\ge \Rank M_{n}$.
 \end{theorem}
\noindent
In \cite{cfm} it was shown that this result does not extend to $\deg r > 2$. The example
in \cite{cfm} concerns an $M_{3}$ that is positive and recursively generated, with $\Card \mathcal{V}_{\beta} = \Rank M_{3}$,
but which has no measure. In this example, there is no measure because $L_{\beta}$ is not consistent.
The results of \cite{f1} show that positivity, the variety condition, and consistency are still not sufficient
for representing measures, as we next describe.

For $M_{n}\succeq 0$, consider the
  \textit{ $(y-x^{3})$-pure} case, when  the column dependence relations in
$M_{n}$ are precisely those given by $Y=X^{3}$, recursiveness, and linearity, i.e.,
column relations of the form $(s(x,y)(y-x^{3}))(X,Y) = 0$    ($\deg s\le n-3$). Thus $M_{n}$ is positive,
 $\Rank M_{n} \le \Card \mathcal{V}(M_{n})$ ($=\Card \Gamma = +\infty$), and it follows from Lemma 3.1 in \cite{f1} that $M_{n}$
is consistent.
In \cite{f1} we described a particular, easily computable,  rational expression in the moment data, $\psi$,
and solved the $(y-x^{3})$-pure TKMP as follows.

\begin{theorem}\label{f1thm}
If $M_{n}\succeq 0$ is $(y-x^{3})$-pure, then $\beta$ has a representing measure if and only if
$\beta_{1,2n-1}> \psi$.
\end{theorem}
\noindent In the proof of Theorem \ref{f1thm}, the numerical test $\beta_{1,2n-1}> \psi$ leads to a flat extension $M_{n+1}$.
By contrast with this result, the other existence results in \cite{f1,z4}
generally presuppose the existence of a certain positive moment matrix extension of $M_{n}$,
but do not give an explicit test for the extension.
The proof of Theorem \ref{f1thm} in \cite{f1}  is quite lengthy.
In the sequel we will use the \textit{core variety} to present a shorter, more transparent proof.
This approach also provides a core variety framework for studying the $(y-x^{d})$-pure truncated moment problem.

The core variety provides an approach to
establishing the existence of
representing measures based on
methods of convex analysis. For the polynomial case,
this was introduced in \cite{f3}, and some of the ideas go back to \cite{fn}.
The discussion below is based on joint work of the first author with G.\ Blekherman \cite{bf},
which treats general Borel measurable functions, although here we only
require polynomials.

Given $\beta \equiv \beta^{(2n)}$ and its Riesz functional
$L \equiv L_{\beta}$,
define $V_{0} := \mathcal{V}(M_{n})$ and for $i\ge 0$, define
$$V_{i+1} :=
\displaystyle \bigcap_{\substack{f\in \ker L,\\  f|V_{i} \ge 0}} \mathcal{Z}(f),$$
where $\mathcal{Z}(f)$ denotes the set of zeros of $f(x,y)$ in $\mathbb{R}^{2}$
(or, equivalently, in $V_{i}$).

We define the
\textit{core variety} of $L$ by
$$\mathcal{CV}(L):= \displaystyle \bigcap_{i\ge 0} V_{i}.$$

\begin{prop}
[{\cite{f3}}]
\label{coreprop}
 \noindent If $\mu$ is a representing measure for $L$, then
 $\supp\mu
\subseteq \mathcal{CV}(L)$.
\end{prop}

\noindent If $\mu$ is a representing measure,
then
$$\Rank M_{n}(\beta) \le \Card(\supp\mu) \le \Card  \mathcal{CV}(L_{\beta})\le \Card V_{i}
\quad
(\text{for every }i\ge 0).$$
We thus have the following test for the nonexistence of  representing measures.
\begin{cor}
[{\cite{f3}}]
\label{coretest}
If $\Card V_{i} < \Rank M_{n}$ for some $i$, then $\beta$ has no representing measure.
\end{cor}

Proposition \ref{coreprop} shows that if $\beta$ has a representing measure, then $\mathcal{CV}(L)$
is nonempty. The main result concerning the core variety is the following converse.
\begin{theorem}
[{\cite[Core Variety Theorem]{bf}}]
\label{cvt}
 $L\equiv L_\beta$ has a representing measure if and only if $\mathcal{CV}(L)$ is nonempty.
In this case, $\mathcal{CV}(L)$ coincides with the union of supports of
all finitely atomic representing measures for $L$.
\end{theorem}

In view of Proposition \ref{coreprop}, $\mathcal{CV}(L)$ is also the union of supports
of all representing measures.
In general, it may be difficult to compute the core variety,
due to the difficulty of characterizing the nonnegative polynomials on
$V_{0}$, $V_{1}$, $V_{2},\ldots$, but Theorem \ref{cvt} leads to the following criterion for
stability.

\begin{prop}
[{\cite{bf}}]
\label{prop:intro}
If $V_{k}$ is finite, then $\mathcal{CV}(L) = V_{k}$ or $\mathcal{CV}(L) = V_{k+1}$.
\end{prop}

In the $(y-x^{d})$--pure case for $M_n(\beta)$, $V_0$ is clearly the curve
$y=x^{d}$. Since $y-x^d$ is irreducible and $\mathcal{CV}(L)$ is an algebraic set, it follows
that either $V_1 = V_0$ ($=\mathcal{CV}(L)$), or $V_1$ is finite and
Proposition \ref{prop:intro} implies $\mathcal{CV}(L) = V_1$ or $\mathcal{CV}(L) = V_2$. We conclude this section by noting the case when
$M_n(\beta)$ has the $Y =X^d$ column relation but is \textit{not} $(y-x^d)$--pure.
In this case there is a column relation
$g(X,Y) =0$, where $g(x,y)$ is not
a multiple of $f(x,y) := y-x^d$. Since $f$ is irreducible, it follows that
$f$ and $g$ are relatively prime, so Bezout's Theorem implies        that
$\Card \mathcal{CV}(L) \leq \Card V_0 \leq \deg f \cdot \deg g$.
Examples computing $\mathcal{CV}(L)$
in the finite-variety case can be found in \cite{f3}.

\setcounter{equation}{1}
\section{
A core variety approach to the pure $Y=X^{d}$ moment problem.}
\label{Section 2}

    Throughout this section assume $n\geq d$.
 Suppose $M_{n}(\beta)$ is positive semidefinite and $(y-x^{d})$-pure, i.e.,
the column dependence relations in $M_{n}$ are precisely the linear combinations of the column relations
\begin{equation}\label{relations}
X^{r}Y^{s+1} = X^{r+d}Y^{s} \quad{\rm{for}}~ r,s \ge 0,~ r+s \le n-d.
\end{equation}

In this section we introduce a \textit{core matrix} $C$ associated to $\beta$; the main result of this section,
Theorem \ref{Jmeasure},
{characterizes the existence
of representing measures for $\beta$
in terms of the positivity properties of $C$ and ``recursiveness'' in its kernel.
Using the Core Variety Theorem
we show that the union of supports of all representing measures is
the curve
$$\Gamma := \mathcal{Z}(y-x^{d}) = \{(x,x^{d}):~ x\in \mathbb{R}\}$$
if and only if there is a positive definite completion of the core matrix.}

{Namely, we employ the connection
between} the existence of representing measures for
$\beta \equiv \beta^{(2n)}$
and
the core variety of the Riesz functional $L \equiv L_{\beta}$.

Setting
$V_{0} = \mathcal{V}(M_{n})= \Gamma$,
we seek to compute
$$V_{1} := \mathcal{Z}(p\in \ker L\colon p|V_{0} \ge 0),$$
the common zeros of the polynomials in $\ker L$ that are nonnegative on $V_{0}$.
To this end, we require a concrete description of $\ker L$.

\begin{lem}
\label{Bbasis}
Suppose $M_{n}(\beta)$
satisfies column relations \eqref{relations}.
Then the polynomials
\begin{align*}
f_{ij}(x,y)
&= x^{i}y^{j} - \beta_{ij}&&\text{for }0\le i < d , j\ge 0, \text{ and }0< i+j\le 2n,\\
g_{kl}(x,y)
&= (y-x^{d})x^{k}y^{l}&&\text{for }k,l \ge 0,~ k+l\le 2n-d.
\end{align*}
form a basis $\cB$ for $\ker L_{\beta}$.

Conversely, let $L:\cP_{2n}\to \RR$ be a linear functional such that $\cB$
is a basis for $\ker L$. Then the moment matrix $M_n(\beta)$ of the sequence $\beta$, such that $L=L_{\beta}$, satisfies column relations \eqref{relations}.
\end{lem}

\begin{remark}

In the statement of Lemma \ref{Bbasis}, $M_n(\beta)$ does not have to be $(y-x^d)$-pure for $\cB$ to be the basis for $\ker L_\beta$. There may be column relations other than the linear combinations of \eqref{relations}, but $\cB$ will still be a basis. Another choice of a basis for $\ker L_\beta$, which works for any sequence $\beta$, is
$\{f_{ij}\}$ for $0\leq i,j$, $0<i+j\leq 2n$, where $f_{ij}$ are defined as in the statement of the lemma. However, this basis tells us nothing about the column relations of $M_n(\beta)$.
To explicitly determine column relations from the basis for $\ker L_\beta$, in addition to a ``good'' choice of the basis, the rank of $M_n(\beta)$ must also be given.
\end{remark}

\begin{proof}[Proof of Lemma \ref{Bbasis}] Clearly, each $f_{ij}\in \ker L_{\beta}$. For $k,l\ge 0$ with $k+l\le 2n-d$, $g_{kl}\in \mathcal{P}_{2n}$.
If $k+l\le n$, then
 $$L_{\beta}(g_{kl}) = \langle M_{n}\widehat{(y-x^{d})}, \widehat{x^{k}y^{l}}\rangle
= \langle M_{n}\widehat{y} -M_n \widehat{x^{d}}, \widehat{x^{k}y^{l}}\rangle  = 0,$$
so $g_{kl} \in \ker L_{\beta}$
in this case. In the remaining case,  $n< k+l \le 2n-d$, so there exist integers $r,s,t,u \ge 0$
such that $r+t = k$, $s+u = l$, $r+s = n-d$, and thus $t+u = (k+l)-(r+s)\le 2n-d -(n-d)=n$.
Now
\begin{align*}
L_{\beta}(g_{kl})
&=
L_\beta((y-x^{d})x^{r}y^{s}\cdot x^{t}y^{u})
=\langle M_{n}\widehat{(y-x^{d})x^{r}y^{s}}, \widehat{x^{t}y^{u}}\rangle\\
&=
\langle
M_{n}\widehat{x^{r}y^{s+1}}-M_n\widehat{x^{d+r}y^{s}}, \widehat{x^{t}y^{u}}\rangle,
\end{align*}
so (\ref{relations}) implies $L_{\beta}(g_{kl})=0$.

To show that $\mathcal{B}$ is a linearly independent set of elements of $\mathcal{P}_{2n}$,
suppose $\{a_{ij}\}$ and $\{b_{kl}\}$ are sequences of real scalars (indexed as in the statement of the lemma) such that
in $\mathcal{P}_{2n}$,
\begin{equation}\label{indep}
\displaystyle \sum_{\substack{ 0\le i<d,\; j\ge 0,\\ 0<i+j\le 2n }} a_{ij}f_{ij} +
\displaystyle  \sum_{\substack{k,l\ge 0,\\k+l\le 2n- d }} b_{kl}g_{kl} = 0.
\end{equation}
Plugging $y = x^{d}$ in (\ref{indep}), it follows that
\begin{equation}
\label{indep-v2}
\displaystyle  \sum_{\substack{ 0\le i<d,\; j\ge 0,\\ 0<i+j\le 2n }} a_{ij}(x^{i+dj}-\beta_{ij})\equiv 0
\end{equation}

Suppose that  $0\le i,i' < d$, $j,j'\ge 0$,   $0<i+j,i'+j' \le 2n$
and $i+dj=i'+dj'$. Then $|i -i'|=d|j-j'|$, and since $|i-i'|<d$, it follows that $j=j'$ and  $i=i'$.
Thus, the $x$-exponents appearing in \eqref{indep-v2}
 are distinct,
and since \eqref{indep-v2} holds for every real $x$, it follows that each $a_{ij}=0$.
Now (\ref{indep}) implies
$$
\displaystyle  \sum_{\substack{k,l\ge 0,\\k+l\le 2n- d }} b_{kl}x^{k}y^{l}(y-x^{d}) \equiv 0
$$

Thus, for $y\not = x^{d}$, $ \sum b_{kl}x^{k}y^{l}=0$, so by continuity we have
$ \sum b_{kl}x^{k}y^{l}=0$ for all $x,y\in \mathbb{R}$. It now follows that each $b_{kl} = 0$,
so $\mathcal{B}$ is linearly independent.

Next we show that $\cB$ spans $\ker L_\beta.$
We need to prove that
$\Card \mathcal{B}=\dim \cP_{2n}-1$ ($= \dim\ker L_{\beta}$).
Recall that
$\dim\mathcal{P}_{2n} = \frac{(2n+1)(2n+2)}{2}$.
Note that
$\mathcal{B}$ is the disjoint union of the sets
$\mathcal{C}$ and $\mathcal{D}$, consisting of all $f_{ij}$ and $g_{kl}$ from the lemma, respectively.
Clearly, $\Card \mathcal{D} = \dim \mathcal{P}_{2n-d}
=\frac{(2n-d+1)(2n-d+2)}{2}$.
To compute $\Card \mathcal{C}$,
notice that $\Card \mathcal{C}=\Card \mathcal{E}$,
where
$\mathcal{E}$ is the index set equal to
\begin{align*}
\mathcal E
&:= \{(i,j)\colon 0\le i<d,\ j\ge 0,\ 0<i+j\le 2n\}\\
&=\{\underbrace{(0,1),\ldots,(0,2n)}_{i=0},
\underbrace{(1,0),\ldots,(1,2n-1)}_{i=1},\ldots,\\
&\hspace{2em}\underbrace{(d-1,0),\ldots,(d-1,2n-d+1)}_{i=d-1}\}.
\end{align*}

It follows that
\begin{align*}
\Card \cC
&=\Card \mathcal E
=2n+2n+(2n-1)+\ldots+(2n-d+2)\\
&=-1+\sum_{i=0}^{d-1} (2n+1-i)
=-1+\sum_{i=1}^{2n+1}i- \sum_{i=1}^{2n-d+1} i\\
&=-1+\frac{(2n+1)(2n+2)}{2}-\frac{(2n-d+1)(2n-d+2)}{2}\\
&=-1+\Card \cP_{2n}-\Card \cD,
\end{align*}
whence
$$
\Card \mathcal{B}
= \Card \mathcal{C} + \Card \mathcal{D}
=-1+\Card \cP_{2n},
$$
which shows that $\cB$ is a basis for $\ker L_\beta$.

The converse part is clear.
Namely, $L$ determines the sequence $\beta$ by
$\beta_{ij}=L(x^iy^j)$ for $0\leq i,j$, $i+j\leq 2n$.
(Note that by $f_{ij}\in \ker L$ for $0\le i < d $, $j\ge 0$, and $0< i+j\le 2n$,
for these indices the $\beta_{ij}$ are precisely the constant terms in
the respective $f_{ij}$.)
Recall from \eqref{def:evaluation-v2}--\eqref{def:evaluation} that for $q \in \cP_n$,
$$q(X,Y) = 0 \;\Longleftrightarrow\; L_\beta(qx^iy^j) = 0 \quad \text{for all }i,j\geq 0,\; i+j\leq n.$$

Since
$g_{kl}\in \ker L$
for $k,l\geq 0$, $k+l\leq 2n-d$, it now follows that all  of the relations of \eqref{relations}, as well as their linear combinations, are
column relations of $M_n(\beta)$.
\end{proof}

Returning to the computation of $V_1$, suppose $p\in \ker L$ satisfies $p|\Gamma \ge 0$, i.e., $p(x,x^{d})\ge 0~\forall x\in \mathbb{R}$.
From Lemma \ref{Bbasis}, we may write
\begin{equation}\label{FG}
p = F+G \equiv
\displaystyle \sum_{\substack{ 0\le i<d,\; j\ge 0,\\ 0<i+j\le 2n }} a_{ij}f_{ij} +
\displaystyle  \sum_{\substack{k,l\ge 0,\\k+l\le 2n-d }} b_{kl}g_{kl}.
\end{equation}
Since $p|\Gamma \ge 0$ and  $G|\Gamma \equiv 0$, then
\begin{equation}
 Q(x):= F(x,x^{d}) = \sum_{\substack{ 0\le i<d,\; j\ge 0,\\ 0<i+j\le 2n }} a_{ij}(x^{i+dj}-\beta_{ij})
\end{equation}
satisfies $Q(x)\ge 0~\forall
x\in \mathbb{R}$. Since $\deg Q \le 2nd$, there exist
$$
\widehat{r}\equiv (r_{0},\ldots,r_{nd})\in\mathbb{R}^{nd+1},
\quad
\widehat{s} \equiv (s_{0},
\ldots,s_{nd}) \in \mathbb{R}^{nd+1}$$
such that
\begin{align}
\begin{split}
R(x) &:= r_{0}+ r_{1}x + \cdots + r_{nd}x^{nd},\\
S(x) &:= s_{0}+ s_{1}x + \cdots + s_{nd}x^{nd}
\end{split}
\end{align}
satisfy
\begin{equation}
    \label{eq:Q-R-S-v2}
        Q(x) = R(x)^{2} + S(x)^{2}.
\end{equation}
In the sequel (and moreso in Part 2 \cite{fz})
we will require detailed information about the coefficients of $F$, $R$ and $S$. By comparing coefficients on both sides of \eqref{eq:Q-R-S}, we see that each $a_{ij}$, which is the coefficient in $Q$ of $x^{i+dj}$,
admits a unique expression
 as a homogeneous quadratic polynomial in the $r_{k}$ and $s_{l}$.
Indeed,

\begin{equation}
\label{eq:Q-R-S}
    a_{ij}=
    h_{i,j}(\widehat{r},\widehat{s}):=
     \displaystyle \sum_{\substack{0\le k,l \le nd,\\
0< k+l = i+dj}} r_{k}r_{l} +s_{k}s_{l},
\end{equation}
for $i,j\geq 0$, $i<d$, $i+j\leq 2n$.
Moreover, a comparison of the constant terms in \eqref{eq:Q-R-S-v2} gives
\begin{equation}
\label{eq:cons-term}
    -\sum_{\substack{ 0\le i<d,\; j\ge 0,\\ 0<i+j\le 2n }}
        a_{ij}\beta_{ij}=r_0^2+s_0^2.
\end{equation}

Note also that if $i,j \geq 0$, $i < d$, $i+dj \leq 2nd$, but $i+j >2n$, then since there is
no moment $\beta_{ij}$, the coefficient of $x^{i+dj}$ in $R(x)^2 + S(x)^2$ must be 0.
Let $\cF$ denote the set of all such pairs $(i,j)$.
It is convenient to extend the definition of $h_{i,j}$
in \eqref{eq:Q-R-S} to include these cases, together with the requirements
\begin{equation}
\label{eq:cons-terms}
    0=h_{i,j}(\widehat{r},\widehat{s})
    \quad \text{whenever }(i,j)\in \cF.
\end{equation}
We call each such requirement an \textit{auxiliary requirement}.
Also, we introduce an arbitrary constant $A_{ij}$ for each $(i,j) \in \cF$
to denote the moment $\beta_{ij}$, which is not present in $\beta^{(2n)}$.
We refer to $A_{ij}$ as an \textit{auxiliary moment}.

In the sequel (particularly in Part 2 \cite{fz}) we require the number and location of the $A_{ij}$. To this end, note that:
\begin{align}
\begin{split}
\cF
&:= \{(i,j)\colon i,j \ge 0,\; i<d,\;i+dj \le 2nd,\;
    i+j >2n\}\\
&= \{(i,j)\colon 2n-(d-2)\leq j\leq 2n-1,\; 2n+1-j\leq i\leq d-1\}\\
&=
\bigcup_{j=1}^{d-2}\cF_{j}
\end{split}
\end{align}
where each ${\cF}_j$ is equal to
$$
{\cF}_j
=
\left\{
\begin{array}{rl}
    \{(j+1,2n-j),\ldots,(d-1,2n-j)\},&
        \text{if }j+1\leq d-1,\\[0.3em]
        \emptyset,& \text{otherwise}.
\end{array}
\right.
$$

Hence, $\Card \cF=\sum_{i=1}^{d-2}i=\frac{(d-1)(d-2)}{2}$.
Note that $\mathcal{F}=\emptyset$ for $d=1,2$

\begin{example}
\label{ex-n-d-3}
Let $n=d=3$.
Then $Q$ (cf.\ \eqref{eq:Q-R-S}) is of the form
$$
Q(x)
=
\sum_{\substack{ 0\le i<3,\; j\ge 0,\\ 0<i+j\le 6 }} a_{ij}(x^{i+3j}-\beta_{ij})
=:\sum_{\ell=0}^{18} q_\ell x^{\ell}\in \cP_{18}.
$$

To illustrate \eqref{eq:Q-R-S}, note that $q_4$, which is equal to $a_{11}$, may be expressed as
\begin{align*}
h_{1,1}(\widehat{r},\widehat{s})
&=r_{0}r_{4}+r_{1}r_{3}+r_{2}^{2}+r_{3}r_{1}+r_{4}r_{0}\\
&\quad+s_{0}s_{4}+s_{1}s_{3}+s_{2}^{2}+s_{3}s_{1}+s_{4}s_{0}\\
&=2(r_{0}r_{4}+s_{0}s_{4}+r_{1}r_{3}+s_{1}s_{3})+r_{2}^{2}+s_{2}^{2}.
\end{align*}

Note that $\cF=\{(2,5)\}$,
since for $i=2$ and $j=5$, we have $i+3j = 17 < 2nd=18$, but $7=i+j >2n=6$. Thus $x^{17}$ does not appear
in $Q(x)$, so, from \eqref{eq:cons-terms}, using $ h_{2,5}(\widehat{r},\widehat{s}) = 
2(r_{8}r_{9} + s_{8}s_{9})
$, it follows that
$0=r_{8}r_{9} + s_{8}s_{9}=q_{17}$.
The auxiliary moment in this case is
$\beta_{2,5}$,
which we denote by
$A_{2,5}$.

For $d=3$ and arbitrary $n\in \NN$, which we study in Section \ref{Section 4},
we have $\cF=\{(2,2n-1)\}$ and
the auxiliary requirement (cf.\ \eqref{eq:cons-terms})
is equal to
\begin{equation}\label{d3requirement}
0=h_{2,2n-1}(\widehat{r},\widehat{s}) = 2(r_{3n}r_{3n-1}  + s_{3n}s_{3n-1}),
\end{equation}
with the ``missing" monomial in $Q(x)$ being $x^{2+3(2n-1)}=x^{6n-1}$.
\end{example}

We next introduce the \textit{core matrix} \(C\equiv C_{\beta}\). Its construction serves two related purposes. First, it converts the constant-term condition \eqref{eq:cons-term} into the inner-product identity \eqref{inner-product-expression}. Second, condition \eqref{Jmeasure-pt3} of Theorem \ref{Jmeasure} characterizes the existence of representing measures for \(\beta\) in terms of positive semidefinite, recursively generated completions of \(C\).

Indeed, the column relation \(Y=X^{d}\), together with recursiveness, allows us to reduce the \(X\)-exponent of every monomial modulo \(d\), i.e., 
$
X^{r+d}Y^{s}=X^{r}Y^{s+1}.
$
Consequently, rows and columns indexed by monomials containing a factor \(X^{d}\) are redundant; for example, \(X^{d+1}=XY\) and \(X^{d}Y=Y^{2}\). The core matrix organizes the resulting nonredundant moment data, together with the necessary auxiliary moments, in a Hankel form adapted to both \eqref{inner-product-expression} and the criterion in Theorem \ref{Jmeasure}.

 Let $\lfloor \cdot \rfloor$ denote the floor function. Namely, $\lfloor k\rfloor$ is the greatest integer not larger than $k$.
 For $1\leq i,j\leq dn+1$, let
\begin{equation}\label{KL}
K_{ij}:= (i+j-2)\;\text{mod}\; d
\quad\text{and}\quad
 L_{ij} :=  \lfloor (i+j-2)/d \rfloor ,
\end{equation}
 The
core matrix, a
$(dn+1) \times (dn+1)$ matrix,
is defined by
\begin{equation}\label{J}
C\equiv (C_{ij})_{i,j=1}^{dn+1}:=(\beta_{K_{ij},L_{ij}})_{i,j=1}^{dn+1}.
\end{equation}
However,  if
$\beta_{K_{ij},L_{ij}}$ is
  an auxiliary moment because
  $(K_{ij},L_{ij})\in \cF$,
  we redefine $\beta_{K_{ij},L_{ij}}$ as
  $$\beta_{K_{ij},L_{ij}} := A_{K_{ij},L_{ij}},$$
where $A_{K_{ij},L_{ij}}$ is an arbitrary constant.

To emphasize the dependence of $C$ on the choice of the
constants $A_{ij}$ for $(i,j)\in \cF$,
we sometimes denote $C$ by
$$C\equiv C[\{A_{ij}\}_{(i,j)\in \cF}].$$
  From (\ref{J}), $C$ is clearly a Hankel matrix.\\

For the reader's convenience, we summarize the notation and the principal steps of the construction in the following remark.

\begin{remark}\label{scheme}
The construction of the core matrix may be summarized schematically as follows:
$$
\beta
\longrightarrow \gamma(A)
\longrightarrow C[A]
\longrightarrow \widetilde{\mu}
\overset{\tau_{\#}}{\longrightarrow}\mu,
\qquad
\tau(t)=(t,t^{d}).
$$
Here the notation has the following meaning:
\begin{itemize}
\item \(\mathcal{F}\) is the set of bivariate indices \((i,j)\) for which the moment \(\beta_{ij}\) is not contained in the given sequence \(\beta^{(2n)}\); each such pair corresponds to a gap at the univariate index \(i+dj\).
\item
\(A=\{A_{ij}\}_{(i,j)\in\mathcal{F}}\) denotes a choice of auxiliary bivariate moments. The completed univariate sequence \(\gamma(A)\) is defined by
\[
\gamma_{i+dj}(A)
=
\begin{cases}
    \beta_{ij},& i+j\leq 2n,\\[0.3em]
    A_{ij},& (i,j)\in\mathcal{F},
\end{cases}
\qquad
0\leq i<d,\quad i+dj\leq 2nd.
\]

\item The indices
$I_k=(k-1)\;\operatorname{mod}\;d$,
$J_k=\left\lfloor\frac{k-1}{d}\right\rfloor$
identify the monomial \(x^{I_k}y^{J_k}\) indexing the \(k\)-th row or column in the ordered monomial basis used in \eqref{ordered-basis-for-core} below.

\item The indices
$K_{ij}=(i+j-2)\;\operatorname{mod}\;d$,
$L_{ij}=\left\lfloor\frac{i+j-2}{d}\right\rfloor$
satisfy
$
K_{ij}+dL_{ij}=i+j-2
$
and identify the bivariate moment occupying the \((i,j)\)-entry of the core matrix.

\item The completed core matrix is the Hankel matrix
\[
C[A]
=
\bigl(\gamma_{r+s}(A)\bigr)_{0\leq r,s\leq nd}.
\]
Equivalently,
\[
C[A]_{ij}
=
\gamma_{i+j-2}(A)
=
\begin{cases}
    \beta_{K_{ij},L_{ij}},
        &(K_{ij},L_{ij})\notin\mathcal{F},\\[0.3em]
    A_{K_{ij},L_{ij}},
        &(K_{ij},L_{ij})\in\mathcal{F}.
\end{cases}
\]

\item When \(C[A]\) is positive semidefinite and recursively generated, the sequence \(\gamma(A)\) admits a finitely atomic representing measure \(\widetilde{\mu}\) on \(\mathbb{R}\). Its push-forward
$
\mu=\tau_{\#}\widetilde{\mu}
$
is a representing measure for \(\beta\) supported on
$
\Gamma=\{(t,t^{d}):t\in\mathbb{R}\}.
$
\end{itemize}
\end{remark}

\begin{example}
    \label{ex:core-4}
    For $n=d=4$ the core matrix
    $$C\equiv C[\mathbf{A}_{3,2n-2}, \mathbf{A}_{2,2n-1}, \mathbf{A}_{3,2n-1}]$$
    is the following
    \begin{equation*}
    \resizebox{\linewidth}{!}{$\begin{pmatrix}
                 \beta_{00} & \beta_{10} & \beta_{20} & \beta_{30} & \beta_{01}& \beta_{11} & \beta_{21} & \beta_{31} & \beta_{02} & \beta_{12} & \beta_{22} & \beta_{32} & \beta_{03} & \beta_{13} & \beta_{23} & \beta_{33} & \beta_{04}\\
            \beta_{10} & \beta_{20} & \beta_{30} & \beta_{01} & \beta_{11}& \beta_{21} & \beta_{31} & \beta_{02} & \beta_{12} & \beta_{22} & \beta_{32} & \beta_{03} & \beta_{13} & \beta_{23} & \beta_{33} & \beta_{04} & \beta_{14}  \\
            \beta_{20} & \beta_{30} & \beta_{01} & \beta_{11} & \beta_{21}& \beta_{31} & \beta_{02} & \beta_{12} & \beta_{22} & \beta_{32} & \beta_{03} & \beta_{13} & \beta_{23} & \beta_{33} & \beta_{04} & \beta_{14} & \beta_{24}  \\
            \beta_{30} & \beta_{01} & \beta_{11} & \beta_{21} & \beta_{31}& \beta_{02} & \beta_{12} & \beta_{22} & \beta_{32} & \beta_{03} & \beta_{13} & \beta_{23} & \beta_{33} & \beta_{04} & \beta_{14} & \beta_{24} & \beta_{34}  \\
            \beta_{01} & \beta_{11} & \beta_{21} & \beta_{31} & \beta_{02}& \beta_{12} & \beta_{22} & \beta_{32} & \beta_{03} & \beta_{13} & \beta_{23} & \beta_{33} & \beta_{04} & \beta_{14} & \beta_{24} & \beta_{34} & \beta_{05}  \\
            \beta_{11} & \beta_{21} & \beta_{31} & \beta_{02} & \beta_{12}& \beta_{22} & \beta_{32} & \beta_{03} & \beta_{13} & \beta_{23} & \beta_{33} & \beta_{04} & \beta_{14} & \beta_{24} & \beta_{34} & \beta_{05} & \beta_{15}  \\
            \beta_{21} & \beta_{31} & \beta_{02} & \beta_{12} & \beta_{22}& \beta_{32} & \beta_{03} & \beta_{13} & \beta_{23} & \beta_{33} & \beta_{04} & \beta_{14} & \beta_{24} & \beta_{34} & \beta_{05} & \beta_{15} & \beta_{25}  \\
            \beta_{31} & \beta_{02} & \beta_{12} & \beta_{22} & \beta_{32}& \beta_{03} & \beta_{13} & \beta_{23} & \beta_{33} & \beta_{04} & \beta_{14} & \beta_{24} & \beta_{34} & \beta_{05} & \beta_{15} & \beta_{25} & \beta_{35}  \\
            \beta_{02} & \beta_{12} & \beta_{22} & \beta_{32} & \beta_{03}& \beta_{13} & \beta_{23} & \beta_{33} & \beta_{04} & \beta_{14} & \beta_{24} & \beta_{34} & \beta_{05} & \beta_{15} & \beta_{25} & \beta_{35} & \beta_{06}  \\
            \beta_{12} & \beta_{22} & \beta_{32} & \beta_{03} & \beta_{13}& \beta_{23} & \beta_{33} & \beta_{04} & \beta_{14} & \beta_{24} & \beta_{34} & \beta_{05} & \beta_{15} & \beta_{25} & \beta_{35} & \beta_{06} & \beta_{16}  \\
            \beta_{22} & \beta_{32} & \beta_{03} & \beta_{13} & \beta_{23}& \beta_{33} & \beta_{04} & \beta_{14} & \beta_{24} & \beta_{34} & \beta_{05} & \beta_{15} & \beta_{25} & \beta_{35} & \beta_{06} & \beta_{16} & \beta_{26}  \\
            \beta_{32} & \beta_{03} & \beta_{13} & \beta_{23} & \beta_{33}& \beta_{04} & \beta_{14} & \beta_{24} & \beta_{34} & \beta_{05} & \beta_{15} & \beta_{25} & \beta_{35} & \beta_{06} & \beta_{16} & \beta_{26} & \mathbf{A}{36}  \\
            \beta_{03} & \beta_{13} & \beta_{23} & \beta_{33} & \beta_{04}& \beta_{14} & \beta_{24} & \beta_{34} & \beta_{05} & \beta_{15} & \beta_{25} & \beta_{35} & \beta_{06} & \beta_{16} & \beta_{26} & \mathbf{A}{36} & \beta_{07}  \\
            \beta_{13} & \beta_{23} & \beta_{33} & \beta_{04} & \beta_{14}& \beta_{24} & \beta_{34} & \beta_{05} & \beta_{15} & \beta_{25} & \beta_{35} & \beta_{06} & \beta_{16} & \beta_{26} & \mathbf{A}{36} & \beta_{07} & \beta_{17}  \\
            \beta_{23} & \beta_{33} & \beta_{04} & \beta_{14} & \beta_{24}& \beta_{34} & \beta_{05} & \beta_{15} & \beta_{25} & \beta_{35} & \beta_{06} & \beta_{16} & \beta_{26} & \mathbf{A}{36} & \beta_{07} & \beta_{17} & \mathbf{A}{27}  \\
            \beta_{33} & \beta_{04} & \beta_{14} & \beta_{24} & \beta_{34}& \beta_{05} & \beta_{15} & \beta_{25} & \beta_{35} & \beta_{06} & \beta_{16} & \beta_{26} & \mathbf{A}{36} & \beta_{07} & \beta_{17} & \mathbf{A}{27} & \mathbf{A}{37}  \\
            \beta_{04} & \beta_{14} & \beta_{24} & \beta_{34} & \beta_{05}& \beta_{15} & \beta_{25} & \beta_{35} & \beta_{06} & \beta_{16} & \beta_{26} & \mathbf{A}{36} & \beta_{07} & \beta_{17} & \mathbf{A}{27} & \mathbf{A}{37} & \beta_{08}
        \end{pmatrix}$}
    \end{equation*}

\medskip

The rows and columns of $C$
are indexed by the ordered set
    \begin{align}
    \label{ordered-basis-for-core}
    \begin{split}
        &\{\textit{1},X,X^2,X^3,Y,XY,X^2Y,X^3Y, \ldots,Y^k,XY^k,X^2Y^k,X^3Y^k,\ldots,\\
        &\hspace{2cm}
        Y^{n-1},XY^{n-1},X^2Y^{n-1},X^3Y^{n-1},Y^n\}.
    \end{split}
    \end{align}
Note that the columns $X^3Y^{n-2},X^2Y^{n-1},X^3Y^{n-1}$ are not among the columns of $M_n$ but rather of its extension $M_{n+2}$.
So these columns are auxiliary ones in $C$ and contain auxiliary moments.
\end{example}

The next two results provide an alternate description of the core matrix in terms of
moment matrix extensions.
Let $d\geq 2$ and let $M_{n+d-2}$ be some recursively generated extension of
the positive $(y-x^{d})$-pure moment matrix $M_n$.
Let $\widetilde \beta\equiv \widetilde \beta^{(2n+2d-4)}$ be the extended sequence and
let
$L_{\widetilde \beta}:\cP_{2(n+d-2)}\to \RR$  be the corresponding Riesz functional.
Define the
ordered
set of monomials
\begin{align}
\label{def:set-M}
\begin{split}
    \mathcal M
    &:=\{1,x,\ldots,x^{d-1},y,xy,\ldots,x^{d-1}y,
    \ldots, y^{i},xy^{i},\ldots,x^{d-1}y^{i},\\
    &\hspace{3cm} y^{n-1},xy^{n-1},\ldots,x^{d-1}y^{n-1},y^n
\},
\end{split}
\end{align}
and the vector space
\begin{align}
\label{def:vector-spaces}
\begin{split}
\mathcal U
    &:=
    \text{Span}\;
    \{
    s\colon s\in \mathcal M
    \}\subset \cP_{n+d-2},\\
\end{split}
\end{align}

We next define an $(nd+1)\times (nd+1)$ matrix
 $M[\widetilde\beta,\mathcal U]$
 with rows and columns
indexed by the monomials in $\mathcal{M}$ in the order
\begin{equation}
\textit{1},X,\ldots,X^{d-1},Y,XY,\ldots,X^{d-1}Y,\ldots
    Y^{n-1},XY^{n-1},\ldots,X^{d-1}Y^{n-1},Y^n
\end{equation}
i.e.,
for $1\leq k\leq nd+1$, the $k$-th element of this order is equal to $X^{I_k}Y^{J_k}$ where
\begin{equation}
\label{def:Ik-Jk}
    I_k:=(k-1)\;\text{mod}\; d
    \quad \text{and}\quad
    J_k:=\lfloor (k-1)/d\rfloor.
\end{equation}

The
 $(i,j)$-th entry of   $M[\widetilde\beta,\mathcal U]$ is  defined to be
      equal to
\begin{equation}
    \label{eq-3}
    L_{\widetilde \beta}(x^{I_i+I_j}y^{J_i+J_j})
    =
    \widetilde\beta_{I_i+I_j,J_i+J_j}.
\end{equation}

More generally,
for $r,s \in \mathcal U$ (cf.\ \eqref{def:vector-spaces}), with coefficient vectors $\widehat{r},
\widehat{s}$ relative to the
ordered basis of monomials in $\mathcal M$
(cf.\ \eqref{def:set-M}), we have
\begin{equation}
\label{Riesz-to-MM-correspondence}
    \big\langle M[\widetilde \beta,\mathcal U]\;\widehat{r},\widehat{s}\big\rangle
        := L_{\widetilde \beta}(rs).
\end{equation}

\begin{lem}
\label{Riesz-simplified}
For $1\leq i,j\leq nd+1$ the following holds:
\begin{equation}
\label{rg-rel}
    L_{\widetilde \beta}(x^{I_i+I_j}y^{J_i+J_j})=
    \widetilde \beta_{K_{ij},L_{ij}},
\end{equation}
where $K_{ij}$, $L_{ij}$ are as in \eqref{KL}.
\end{lem}

\begin{proof}
We have that
\begin{align}
\label{eq}
\begin{split}
    K_{ij}+dL_{ij}
    &=
    i+j-2
    =I_{i}+I_j+d(J_i+J_j),
\end{split}
\end{align}
where in the second equality we used $i+j-2=(i-1)+(j-1)$.
We separate two cases according to the value of the sum $I_i+I_j$.

\noindent \textbf{Case a):}
$I_i+I_j<d$.

Then \eqref{eq} implies that
$K_{ij}=I_i+I_j$ and $L_{ij}=J_i+J_j$.
Using this

in \eqref{eq-3}, \eqref{rg-rel} follows.

\noindent \textbf{Case b):}
$I_i+I_j\geq d$.

Then \eqref{eq} implies that
\begin{equation}
    \label{Kij-Lij}
        K_{ij}=I_i+I_j-d\quad \text{and} \quad L_{ij}=J_i+J_j+1.
\end{equation}

Since $M_{n+d-2}$ is recursively generated, we have $X^{r+d}Y^{s} =  X^{r}Y^{s+1}$ in the
rows and columns, and therefore $\widetilde{\beta}_{r+d,s} = \widetilde{\beta}_{r,s+1}$.
The assumption of Case b), and \eqref{Kij-Lij} used in \eqref{eq-3}, together with $M_{n+d-2}$ being recursively generated, therefore
 imply that
\begin{align*}
\widetilde{\beta}_{I_i+I_j,J_i+J_j}
&= \widetilde{\beta}_{
    K_{ij}+d,J_{i}+J_j}
=\widetilde{\beta}_{
    K_{ij},J_{i}+J_j+1}
\underbrace{=}_{\eqref{Kij-Lij}}\widetilde{\beta}_{
    K_{ij},L_{ij}}.
\end{align*}
proving \eqref{rg-rel}.
\end{proof}

\begin{proposition}
\label{prop:core-sub}
Assume the notation above. Then:
\begin{enumerate}
    \item\label{prop:core-sub-pt1}
        If the sequence $\widetilde \beta$ has a representing measure, then
        $M[\widetilde \beta,\mathcal U]$ is positive semidefinite.
    \item\label{prop:core-sub-pt2}
        Let
        $\widetilde M[\widetilde \beta,\mathcal U]$ be obtained from
        $M[\widetilde \beta,\mathcal U]$
        by
        replacing each $\widetilde\beta_{ij}$
        satisfying
        $i\;\text{mod}\;d+j+\lfloor\frac{i}{d}\rfloor>2n$ with the auxiliary moment $A_{ij}$. Then
            $$C[\{A_{ij}\}_{(i,j)\in \cF}]=\widetilde M[\widetilde \beta,\mathcal U].$$
\end{enumerate}
\end{proposition}

\begin{proof}
Part \eqref{prop:core-sub-pt1} follows from the equality \eqref{Riesz-to-MM-correspondence} and $L_{\widetilde\beta}(r^2)=\int r^2d\mu\geq 0$, where $\mu$ is a representing measure for $\widetilde\beta$.
For part \eqref{prop:core-sub-pt2}, first note that not all $\widetilde\beta_{ij}$ with $i+j>2n$ are auxiliary moments. By recursive generation,
\[
\widetilde\beta_{ij}=\widetilde\beta_{i-d,j+1}
\qquad (d\le i<2d-1).
\]
Thus, in these cases, $\widetilde\beta_{ij}$ is auxiliary only if
\[
i\;\text{mod}\;d+j+\left\lfloor\frac{i}{d}\right\rfloor
=i-d+j+1>2n.
\]
If $i<d$, this condition reduces to $i+j>2n$. Hence \eqref{prop:core-sub-pt2} follows from \eqref{J} and Lemma \ref{Riesz-simplified}.
\end{proof}

 If $H\equiv {(h_{i+j-1})}_{i,j=1}^m$ is any $m \times m$ Hankel matrix
and $\widehat{t}:= (t_{1},\ldots,t_{m})\in \mathbb{R}^{m}$, then
\begin{equation}\label{hankel}  \langle H\widehat{t},\widehat{t} \rangle
=
\displaystyle \sum_{i=1}^{m} \displaystyle \sum_{j=1}^{m} t_{i}  h_{i+j-1}t_{j}
=
\displaystyle \sum_{k=1}^{2m-1}
\Big(h_{k} \cdot \displaystyle \sum_{\substack{1\le i,j \le m,\\ i+j=k+1}} t_{i}t_{j}\Big).
\end{equation}

\begin{lem}\label{Jrr}
{Let
$\widehat{r}\equiv (r_{0},\ldots,r_{nd})\in \mathbb{R}^{nd+1}$,
$\widehat{s} \equiv (s_{0},
\ldots,s_{nd}) \in \mathbb{R}^{nd+1}$
satisfy \eqref{eq:cons-terms}.
For $i,j\geq 0$, with $i<d$ and {$0<i+j\leq 2n$}, define $a_{ij}$ by \eqref{eq:Q-R-S}.
Then}
\begin{equation}
\label{inner-product-expression}
\langle C\widehat{r}, \widehat{r}\rangle + \langle C\widehat{s}, \widehat{s}\rangle=0
\quad
\Longleftrightarrow
\quad
\eqref{eq:cons-term}\; \text{holds}.
\end{equation}
\end{lem}

\begin{proof}
Let $I_k$, $J_k$ be as in \eqref{def:Ik-Jk}.
Further, let
\begin{equation}
\label{def:h_k}
h_k:=
\left\{
\begin{array}{rl}
    \beta_{I_k,J_k},&    \text{if }I_k+J_k\leq 2n,\\[0.2em]
    A_{I_k,J_k},&    \text{if }I_k+J_k> 2n.
\end{array}
\right.
\end{equation}
We now apply (\ref{hankel}) with $m = nd+1$, $H=C$
with $h_k$ as in \eqref{def:h_k},
and with $t_{p} = r_{p-1}$ or  $t_{p} = s_{p-1}$ ($1\le p \le nd+1$):
\begin{align*}
&\langle C\widehat{r}, \widehat{r} \rangle + \langle C\widehat{s}, \widehat{s} \rangle=\\
&= \displaystyle \sum_{k=1}^{2nd+1}
\Big(h_k \cdot\displaystyle \sum_{\substack{1\le p,q \le nd+1,\\ p+q = k+1}} (r_{p-1}r_{q-1}
+ s_{p-1}s_{q-1})\Big)\\
&=
r_0^2+s_0^2+
\displaystyle
\sum_{\substack{ 0\le i<d,\; j\ge 0,\\ 0<i+j\le 2n }}
\Big(\beta_{ij}\cdot\displaystyle \sum_{\substack{0\le p,q \le nd,\\0<p+q =i+dj}}
(r_{p}r_{q}
+ s_{p}s_{q})\Big)
+\\
&\hspace{3cm}+
\sum_{\substack{ 0\le i<d,\; j\ge 0,\\ i+j> 2n }}
\Big(A_{ij}\cdot\displaystyle \sum_{\substack{0\le p,q \le nd,\\0<p+q =i+dj}}
(r_{p}r_{q}
+ s_{p}s_{q})\Big)
\\
&=r_0^2+s_0^2
+\sum_{\substack{ 0\le i<d,\; j\ge 0,\\ 0<i+j\le 2n }}
a_{ij}\beta_{ij}
\end{align*}
where we used
the assumption of the lemma in the last equality.

Now the equivalence of the lemma easily follows.
\end{proof}

\begin{rem}\label{reverse}
It is important for the sequel to note that
the implication ($\Leftarrow$) of Lemma \ref{Jrr}
may be used
in order to construct elements $p$ of $\ker L$ satisfying $p|\Gamma \ge 0$,
so that $\mathcal{CV}(L) \subseteq \mathcal{Z}(p|\Gamma)$.
For suppose
$\widehat{r},~\widehat{s}\in \mathbb{R}^{nd+1}$ satisfy $h_{i,j}(\widehat{r},\widehat{s})=0$
for every $(i,j)\in \cF$ and
$\langle C\widehat{r}, \widehat{r}\rangle + \langle C\widehat{s}, \widehat{s}\rangle = 0$.
Now define $a_{ij} = h_{ij}(\widehat{r}, \widehat{s})$ ($i,j \ge 0$, $i<d$, $0<i+j\le 2n$).
Then $p:= \sum a_{ij}f_{ij}\in \ker L$ satisfies $p(x,x^{d}) = R(x)^{2}+S(x)^{2}$, where
$R(x) := r_{0}+ r_{1}x + \cdots + r_{nd}x^{dn}$
 and $S(x) := s_{0}+ s_{1}x + \cdots + s_{nd}x^{dn}$.
Now we have
$\mathcal{CV}(L) \subseteq \{(x,x^{d}):  R(x) = S(x) =0\}$ and
$\Card \mathcal{CV}(L) \le \min\{ \deg R, \deg S\}$.
\end{rem}

Let $A\equiv\{A_{ij}\}_{(i,j)\in \cF}$ with $A_{ij}\in \RR$.
We say that the core matrix $C[A]$ is \textit{recursively generated} if
for every $v\in \RR^{nd}$ satisfying
$\begin{pmatrix} v \\ 0 \end{pmatrix}\in \ker C[A],$
it follows that
$\begin{pmatrix} 0 \\ v \end{pmatrix}\in \ker C[A].$

\begin{remark}
\label{remark-rg}
Note that the definition above is equivalent to the definition of a
``recursively generated'' Hankel matrix given in \cite{cf1991}. However, it does not encompass the notion of recursiveness for a general multivariable moment matrix given in item iv) preceding Theorem \ref{degree2}.
\end{remark}

Let $0_{k\times 1}\in \RR^k$ stand for a zero column vector.

By Remark \ref{remark-rg} and properties of recursively generated Hankel matrices \cite{cf1991}, every singular, recursively generated $C[A]$ admits an $r\le nd+1$ and a vector $v:=(v_i)_{i=1}^r\in\RR^r$, with $v_r\neq0$, such that
$$
\ker C[A]
=
\operatorname{span}\Big\{
\begin{pmatrix} v \\ 0_{(nd+1-r)\times 1} \end{pmatrix},
\begin{pmatrix} 0\\ v \\ 0_{(nd-r)\times 1} \end{pmatrix},\ldots,
\begin{pmatrix} 0_{(nd+1-r)\times 1} \\ v \end{pmatrix}
\Big\}.
$$
If we normalize $v$ so that $v_r=1$, then $v$ is uniquely determined. We call this unique $v$ the \textit{generating kernel vector} of $C[A]$.

Because $y-x^d$ is irreducible, the core variety is either the entire curve or a finite set of points in the curve.
{The following theorem characterizes the existence of a representing measure for $\beta$ in terms of the existence of auxiliary moments such that the core matrix is positive and recursively generated. It also characterizes the type of core variety in terms of positive completions of the core matrix.}

\begin{theorem}\label{Jmeasure}
{
Let $\beta\equiv \beta^{(2n)}$ be a given sequence such that
$M_n\equiv M_{n}(\beta)$ is positive semidefinite
and $(y-x^{d})$-pure.
Let $\Gamma := \mathcal{Z}(y-x^{d})$.
The following statements are equivalent:
\begin{enumerate}
\item $\beta$  admits a representing measure (necessarily supported in $\Gamma$).
\smallskip
\item $\beta$  admits a finitely atomic representing measure (necessarily supported in $\Gamma$).
\smallskip
\item\label{Jmeasure-pt3} There exist auxiliary moments $A\equiv \{A_{ij}\}_{(i,j)\in \cF}$, such
that the core matrix $C[A] \equiv C[\{A_{ij}\}_{(i,j)\in \cF}]$ is positive semidefinite and recursively generated.
\end{enumerate}
}
\smallskip

Moreover,
$\mathcal C\mathcal V(L_\beta)=\Gamma$ (and therefore
$\beta$ has finitely atomic representing measures (respectively, representing measures) whose union of supports is $\Gamma$)
if and only if there is some choice of auxiliary moments $A$ such that $C[A]$ is positive definite.
Further, the following are equivalent:
\begin{enumerate}
\setcounter{enumi}{3}
\item
\label{Jmeasure-pt4}
$\mathcal C\mathcal V(L_\beta)$ is a nonempty finite subset of $\Gamma$.
\smallskip
\item
$\beta$ has a unique measure $\mu$, with $\supp\mu = \mathcal{CV}(L_\beta)$.
\smallskip
\item
\label{Jmeasure-pt6}
There is a unique positive semidefinite, recursively generated completion $C[A]$, which is necessarily singular, 
with $\Rank C[A] = \Card\mathcal{CV}(L_\beta)$.
\end{enumerate}
\end{theorem}
\begin{proof}
The equivalence $(i) \Leftrightarrow (ii)$ follows from Richter's Theorem \cite{Ric}
(or by Theorem \ref{cvt}).

Next we establish the implication $(ii) \Rightarrow (iii)$.
Suppose $M_{n}(\beta)$ is $(y-x^{d})$-pure and that $\beta$ has a finitely atomic representing measure $\mu$ supported in $\Gamma$.

Thus, $\mu$ is of the form
\begin{equation}
    \label{f-a-r-m}
    \mu = \displaystyle \sum_{k=1}^{m} a_{k}\delta_{(x_{k},y_{k})},
\end{equation}
where $m>0$, each $a_{k}>0$, and $y_{k}=x_{k}^{d}$ for each $k$.
Since $\mu$ has moments
of all orders, we may consider  the moment matrix $M_{n+t}[\mu]$,
containing $\mu$-moments up to degree $2n+2t$,
 where $t= \left \lceil \frac{d-2}{2} \right \rceil$.
 Here, $\lceil \cdot \rceil$ denotes the ceiling function, i.e., the smallest integer greater than or equal to its argument.
 Using the moment data $\widetilde{\beta}^{(2(n+t))}$ from $M_{n+t}[\mu]$,
i.e., $\widetilde{\beta_{ij}} = \int x^{i}y^{j}d\mu$, ($i,j\ge 0,~i+j\le 2(n+t)$), let
\begin{equation}
\label{univariate-corresponding-sequence}
\gamma_{p} = \widetilde{\beta}_{p\;\text{mod}\; d,\lfloor \frac{p}{d} \rfloor} ~~(0\le p\le 2nd).
\end{equation}
Since $M_{n}[\mu] = M_{n}(\beta)$, we have
$$\gamma_{p} = \beta_{p\;\text{mod}\; d,\lfloor \frac{p}{d} \rfloor}
\quad {\rm if}~0\le p\le 2nd
\quad\text{and}\quad
p\;\text{mod}\; d +\lfloor \frac{p}{d} \rfloor \le 2n.$$

We next show that
\begin{equation}
    \label{f-a-r-m-uni}
        \widetilde{\mu} := \displaystyle \sum_{k=1}^{m} a_{k}\delta_{x_{k}}
\end{equation}
is a representing measure  for $\gamma := \{\gamma_{p}\}_{0\le p\le  2nd}$. Indeed,
for $0\le p \le 2nd$ we have
$$\sum a_{k}x_{k}^{p} = \sum a_{k}x_{k}^{p\;\text{mod}\; d+d \lfloor \frac{p}{d} \rfloor}
= \sum  a_{k}x_{k}^{p\;\text{mod}\;d}y_{k}^{\lfloor \frac{p}{d} \rfloor}
=\widetilde{\beta}_{p\;\text{mod}\; d,\lfloor \frac{p}{d} \rfloor} = \gamma_{p}.$$
It now follows that the moment matrix for $\gamma$, which is the Hankel matrix $H(\gamma)\equiv (\gamma_{i+j})_{0\le i,j \le nd} $,
 is positive semidefinite {and recursively generated} (cf.\ Section \ref{sec:prel}). If, in the core matrix $C[A]$, for each $(i,j) \in \mathcal{F}$ we set $A_{ij} = \gamma_{i+dj} = \widetilde{\beta}_{ij}$, then $C[A]$ coincides with $H(\gamma)$, and is thus positive semidefinite {and recursively generated}. This is precisely $(iii)$.

Next we establish the implication $(iii) \Rightarrow (ii)$.
 Suppose there exist auxiliary moments $A$ such that $C[A]$
 is positive semidefinite and recursively generated.
 We will prove that $\beta$ has a finitely atomic representing measure.
 Define a univariate sequence $\gamma\equiv\{\gamma_{p}\}_{0\leq p\leq 2nd}$
 as in \eqref{univariate-corresponding-sequence} above, where $\widetilde\beta_{ij}$
 is either $\beta_{ij}$ or $A_{ij}$. Since the Hankel matrix
 $H(\gamma)\equiv (\gamma_{i+j})_{0\le i,j \le nd}$ coincides with $C[A]$ (by definition of $\gamma$), it follows that it is positive semidefinite and recursively generated. By \cite[Theorem 3.9]{cf1991}, $\gamma$ has a finitely atomic representing measure
 $\widetilde{\mu} := \displaystyle \sum_{k=1}^{m} a_{k}\delta_{x_{k}}$. But then
 $ \mu = \displaystyle \sum_{k=1}^{m} a_{k}\delta_{(x_{k},y_{k})}$ is a representing measure for $\beta$.
 Indeed, for $0\leq i,j\leq 2n, i+j\leq 2n$ we have
 \begin{align*}
 \sum a_{k}x_{k}^{i}y_k^j
 &=
 \sum a_{k}x_{k}^{i+dj}
 =
 \gamma_{i+dj}
 =
 \beta_{i\;\text{mod}\; d,j+\lfloor \frac{i}{d} \rfloor}
 =
 \beta_{i,j},
\end{align*}
where in the last equality we used that $\beta_{r+d,s}=\beta_{r,s+1}$ for $0\leq r,s$ such that $r+s+d\leq 2n$.

It remains to address the core variety. First assume that $C[A]$ is positive definite for some choice of auxiliary moments $A$.
For $L\equiv L_{\beta}$, we have $V_0=\mathcal V(M_n)=\Gamma$, and we consider
\[
V_1:=\mathcal Z\bigl(p\in\ker L:\ p|_{V_0}\ge0\bigr).
\]
For $p\in\ker L$ with $p|_{V_0}\ge0$, write $p=F+G$ as in \eqref{FG}. The discussion following Lemma \ref{Bbasis} gives
\[
Q(x):=F(x,x^d)=R(x)^2+S(x)^2.
\]
The vectors $\widehat r$ and $\widehat s$ satisfy \eqref{eq:Q-R-S}--\eqref{eq:cons-terms}. Lemma \ref{Jrr} yields
\[
\langle C\widehat r,\widehat r\rangle+\langle C\widehat s,\widehat s\rangle=0.
\]
Since $C$ is positive definite, $\widehat r=\widehat s=0$. Hence \eqref{eq:Q-R-S} gives $a_{ij}=0$ for every $i,j$, so $F=0$. Since $\mathcal Z(G|_{\Gamma})=\Gamma$, we obtain $\mathcal Z(p|_{\Gamma})=\Gamma$. Thus $V_1=V_0=\Gamma$ and $\mathcal{CV}(L)=\Gamma$. The Core Variety Theorem then implies that $\beta$ has finitely atomic representing measures whose union of supports is $\Gamma$.

Assume next that $\mathcal{C}\mathcal{V}(L)=\Gamma$.
We need to prove that there exists a choice of auxiliary moments $A$ such that $C[A]$ is positive definite.
  We first show that if there exist distinct completions $C[A_1]$ and $C[A_2]$ that are positive semidefinite, recursively generated and \textit{singular}, then there is a
positive definite completion $C[A]$.
\bigskip

\noindent\textbf{Claim.}
\( C\!\left[\tfrac{1}{2}A_1 + \tfrac{1}{2}A_2\right] \) is positive definite.

\medskip
\noindent\textit{Proof of Claim.}
We have
\[
C\!\left[\tfrac{1}{2}A_1 + \tfrac{1}{2}A_2\right]
  = \tfrac{1}{2} C[A_1] + \tfrac{1}{2} C[A_2].
\]
Since all three matrices are positive semidefinite, it follows that
\begin{equation}
\label{kernel-equality}
\ker\!\left(C\!\left[\tfrac{1}{2}A_1 + \tfrac{1}{2}A_2\right]\right)
   = \ker(C[A_1]) \cap \ker(C[A_2]).
\end{equation}

Assume that \( C\!\left[\tfrac{1}{2}A_1 + \tfrac{1}{2}A_2\right] \) is not positive definite.
Let \( v \in \mathbb{R}^r \), \( r \le nd+1 \), be its generating kernel vector.
By \eqref{kernel-equality}, the vector
\[
u :=
\begin{pmatrix}
    0_{(nd+1-r)\times 1}\\ v
\end{pmatrix}
\]
lies in
\(
\ker(C[A_1]) \cap \ker(C[A_2]).
\)
Now examine the last column of each of the matrices
\( C[A_1] \), \( C[A_2] \),
\( C\!\left[\tfrac{1}{2}A_1 + \tfrac{1}{2}A_2\right] \),
proceeding from the top to the bottom row.
At the first occurrence of an auxiliary moment, the corresponding entry must be identical
in all three matrices, because \( u \) is a common kernel vector and the other entries in the row of the auxiliary moment coincide in all three matrices.
Proceeding to the second auxiliary moment, we again conclude (using the Hankel structures and that the first auxiliary moments already coincide) that this entry must also agree in all three matrices.
Continuing inductively, we find that all auxiliary moments coincide, i.e.,
\( A_1 = A_2 \).
This contradicts the assumption \( A_1 \neq A_2 \), completing the proof of the claim.\\

Now let $\mu_1$ be a finitely atomic representing measure for $\beta$ as given by \eqref{f-a-r-m}. As in the proof of  $(ii) \Rightarrow (iii)$, we associate to $\mu_1$ the univariate sequence $\gamma_1$ with Hankel matrix $H(\gamma_1)$ and representing measure $\widetilde\mu_1$ as in \eqref{f-a-r-m-uni}, and use these to define the positive semidefinite and recursively generated completion $C[A_1] := H(\gamma_1)$.
If $C[A_1]$ is positive definite, we are done, so we may assume $C[A_1]$ is singular, in which case $\gamma_1$ has a unique representing measure
by \cite[Theorem 3.10]{cf1991}, namely $\widetilde\mu_1$. Since $\mathcal C\mathcal V(L) =\Gamma$, there exists a finitely atomic representing measure $\mu_2$ for $\beta$ that is distinct from
$\mu_1$. As above,
we may associate to $\mu_2$ its univariate sequence $\gamma_2$ with Hankel matrix $H(\gamma_2)$ and representing measure $\widetilde \mu_2$, and use these to define the positive semidefinite and recursively generated completion $C[A_2] := H(\gamma_2)$.
If $C[A_2]$ is positive definite, we are done, so we may assume $C[A_2]$ is singular.
Now, if $C[A_1]$ and $C[A_2]$ are distinct we may apply the Claim to conclude that there is a positive definite completion, as desired. So we may assume that
$C[A_2] = C[A_1]$, whence $H(\gamma_2) = H(\gamma_1)$. Since $H(\gamma_1)$ has the unique representing measure $\widetilde\mu_1$ and $\widetilde\mu_2$ is a representing measure for $H(\gamma_2)$, we have $\widetilde\mu_2 = \widetilde\mu_1$. It follows readily that $\mu_2 = \mu_1$, a contradiction.
Thus $C[A_1]$ and $C[A_2]$ are distinct, which implies that there is a positive definite completion $C[A]$.

The equivalences among $(iv), (v), (vi)$ follow directly from the reasoning above.
\end{proof}

\begin{remark}
Note that if there is a positive definite completion $C[A]$,
there may also be positive semidefinite, recursively generated but singular completions. Thus, for $d =3$,
the set of $A$ for which $C[A]$ is positive definite forms an open interval,  and at the interval endpoints the completions are positive semidefinite, recursively generated, but singular.
\end{remark}

We next record some practical considerations concerning the computational implementation of the criterion in Theorem~3.11.

\begin{remark}\label{pf}
The search for a positive semidefinite completion of the core matrix can be formulated as a linear matrix inequality. Recursive generation can then be checked directly by verifying that
$$
\begin{pmatrix}v\\0\end{pmatrix}\in\ker C[A]
\quad\Longrightarrow\quad
\begin{pmatrix}0\\v\end{pmatrix}\in\ker C[A]
\qquad
\text{for every }v\in\mathbb{R}^{nd}.
$$

Once a singular, recursively generated completion has been found, its generating kernel vector \(v=(v_{1},\ldots,v_{r})\) determines the polynomial
$
q(t)=v_{1}+v_{2}t+\cdots+v_{r}t^{r-1}.
$
The real zeros \(x_{1},\ldots,x_{r-1}\) of \(q\) determine the support of the associated univariate representing measure. Here we use the fact that all zeros of $q$ are real and simple, which follows from the solution to the univariate truncated moment problem \cite{cf1991}. The corresponding weights \(w_{1},\ldots,w_{r-1}\) can be obtained by solving the Vandermonde system
$$
\sum_{\ell=1}^{r-1}w_{\ell}x_{\ell}^{k}
=
\gamma_{k}(A),
\qquad
0\leq k\leq r-2.
$$
The representing measure for \(\beta\) is then obtained by pushing this univariate measure forward under \(t\mapsto(t,t^{d})\).

Because rank and kernel computations can be sensitive near the boundary of the positive semidefinite completion region, numerical conclusions concerning singular completions or boundary points should be verified using exact arithmetic or sufficiently high-precision computations.
\end{remark}

The proof of Theorem \ref{Jmeasure} shows that the moment problem for $\beta^{(2n)}$ (with $M_n(\beta)$ positive semidefinite and
$(y-x^d)$-pure) is equivalent to the moment problem for the univariate sequence $\gamma^{(2nd)}$ with gaps corresponding to $\gamma_p = A_{ij}$ for each auxiliary moment $A_{ij}$ with
$p =i+dj$. The presence of these gaps puts the moment problem for $\gamma$ beyond the scope of the truncated Hamburger problem for a sequence without gaps \cite{cf1991}.
The rest of this paper and its sequel \cite{fz} are primarily devoted to developing \textit{concrete} conditions for the existence or nonexistence of auxiliary moments satisfying condition \eqref{Jmeasure-pt3} of Theorem \ref{Jmeasure}. We conclude this section with examples which illustrate cases where the core variety is either the entire curve $y = x^d$ or is empty.
These examples suggest the following question.

\begin{quest}\label{q1}
Let $\beta \equiv \beta^{(2n)}$ be such that $M_n(\beta)$ is positive semidefinite and
$(y-x^d)$-pure. Is it possible for $\beta$ to have a unique representing measure (cf.\ conditions \eqref{Jmeasure-pt4}-\eqref{Jmeasure-pt6} of Theorem \ref{Jmeasure})?
\end{quest}

In Example \ref{d=1,2} (just below) we show that the answer is negative for
$d = 1$ and $d = 2$. In Section \ref{Section 4}
we prove a negative answer for $d = 3$.
This provides a new proof of Theorem \ref{f1thm}. Nevertheless, for $d = 4$
we can establish a positive answer.
The example illustrating this is beyond the scope of this note, as it requires special techniques, but will appear in \cite{fz}.

In the sequel, for $M_{n}\succeq 0$ and $(y-x^{d})$-pure, we denote by
$\widehat{M_{n}}$ the central compression of $M_{n}$ obtained by deleting all rows
and columns $X^{d+p}Y^{q}$ ($p,q\ge 0$, $p+q\le n-d$). The number of rows and columns in
$\widehat{M_{n}}$ is thus $\dim \mathcal{P}_{n} - \dim \mathcal{P}_{n-d} = \frac{d(2n-d+3)}{2}$.
Since $M_{n}$ is positive and $(y-x^{d})$-pure,
it follows immediately that  $\widehat{M_{n}}$ is positive definite
and
\begin{equation}
        \Rank M_{n} = \Rank \widehat{M_{n}} = \frac{d(2n-d+3)}{2}.
\end{equation}

\begin{example}\label{d=1,2}
i) For $d=1$, we have
$C = \widehat{M_{n}} =
\begin{pmatrix}
     \beta_{0,i+j-2}
\end{pmatrix}_{1\le i,j\le n+1}\succ 0$,
so the existence of representing measures whose union of
supports is the line $y=x$ now follows from
Theorem \ref{Jmeasure}. Alternately, using flat extensions, the existence of measures in this case
 follows from
the solution to the  truncated moment problem on a line in \cite{cf2002}.
\smallskip

\noindent ii) For $d=2$, the core matrix $C$ for $M_{n}$ is $(2n+1)\times (2n+1)$,
with
\begin{equation}\label{d=2} C_{ij} = \beta_{(i+j-2)\;\text{mod}\; 2,\lfloor \frac{i+j-2}{2}\rfloor}.
\end{equation}
In $\widehat{M_n}$, column $j$ is the truncation to $\widehat{M_n}$ of the column indexed by
\[
X^{(j-1)\;\text{mod}\;2}Y^{\lfloor (j-1)/2\rfloor}
\]
in $M_n$. Likewise, row $i$ of $\widehat{M_n}$ is the truncation to $\widehat{M_n}$ of the row indexed by
\[
X^{(i-1)\;\text{mod}\;2}Y^{\lfloor (i-1)/2\rfloor}
\]
in $M_n$.
Thus, using the structure of moment matrices,
we have
\begin{equation}\label{mhatij}
\widehat{M}_{ij} = \beta_{(i-1)\;\text{mod}\; 2 +(j-1)\;\text{mod}\; 2, \lfloor (i-1)/2 \rfloor+\lfloor (j-1)/2 \rfloor}.
\end{equation}
By Proposition \ref{prop:core-sub}, we have $C=\widehat{M_n}\succ0$.
The same conclusion follows directly from \eqref{d=2} and \eqref{mhatij}.

Since  $C$ is positive definite,
Theorem \ref{Jmeasure} now
implies that $\beta$ has representing measures whose union of supports is the parabola $y = x^{2}$.
The existence of representing measures
also follows from the solution to the Parabolic Truncated Moment Problem in
\cite{cf2004}, based on flat extensions.
\end{example}

\begin{remark}
\label{rem:ort-equiv}
Note that for $d\ge 3$,
$C \equiv
C[\{A_{ij}\}_{(i,j)\in \cF}]$ does not coincide with $\widehat{M_n}$.
However, its central compression
$\widehat{C}$, obtained by deleting row $k$ and
column $k$ from $C$ in those cases where row $k$
ends with an auxiliary moment, is orthogonally equivalent to $\widehat{M_n}$, and is therefore positive definite. The details of this will appear in \cite{fz}.

Here we only explain the case $d=3$, since this will be needed in the next section to provide a core variety-based solution to the $(y-x^3)$-pure TMP.  The core matrix $C$ for $M_{n}$ is $(3n+1)\times (3n+1)$
with
\begin{equation}\label{d=3} C_{ij} = \beta_{(i+j-2)\;\text{mod}\; 3,\lfloor \frac{i+j-2}{3}\rfloor}.
\end{equation}
Let $\widehat{C}$ be a $3n\times 3n$ principal submatrix of $C$ obtained by deleting
$nd$-th row and column.
Recall that the rows and columns of $\widehat{M_{n}}$, which is also of size $3n\times 3n$,
are labelled in degree-lexicographic order,
\begin{equation*}
    1, X,Y,X^2,XY,Y^2,X^2Y,XY^2,Y^3,
\ldots, X^{2}Y^{n-2},XY^{n-1},Y^{n}
\end{equation*}
(there is no row or column $X^{i}Y^{j}$ with $i\ge 3$).
Let us permute these to the order
\begin{equation}
\label{new-order}
    1, X,X^2,Y,XY,X^2Y,Y^2,XY^2,X^2Y^2,
    \ldots,Y^{n-1},XY^{n-1},Y^{n}.
\end{equation}
Then there exists a permutation matrix $U$ of size $3n\times 3n$
such that
    $U^T \widehat{M}_nU$
has rows and columns indexed by \eqref{new-order}.
Note that
\begin{align}
    \label{ort-equiv-central}
    \begin{split}
    (U^T \widehat{M}_nU)_{ij}
    &=
    \beta_{(i-1)\;\text{mod}\; 3 +(j-1)\;\text{mod}\; 3, \lfloor (i-1)/3 \rfloor+\lfloor (j-1)/3 \rfloor}\\
    &= \beta_{(i+j-2)\;\text{mod}\; 3,\lfloor \frac{i+j-2}{3}\rfloor}.
    \end{split}
\end{align}
The second equality follows from Lemma \ref{Riesz-simplified}.

Equations \eqref{d=3} and \eqref{ort-equiv-central} show that $\widehat C$ is orthogonally equivalent to $\widehat{M}_n$.
\end{remark}

We conclude this section with some examples that illustrate Theorem \ref{Jmeasure}
for a positive semidefinite $(y-x^{d})$-pure $M_{n}(\beta)$.

 Let $\widehat{C}$ denote the compression of
$C\equiv C[A]$ obtained by deleting each row and each column of $C$ that ends in some
auxiliary moment $A_{ij}$.

In the sequel, for $1\le k\le dn+1$, $C_{k}$ denotes the compression of $C$ to the first $k$
rows and columns.

\begin{example}\label{ex:y=x3}
Consider the moment matrix \(M_{4}(\beta)\) given by
\begin{equation*}
\resizebox{\linewidth}{!}{%
$\begin{pmatrix}
1 & 0 & 0 & 1 & 2 & 5 & 0 & 0 & 0 & 0 & 2 & 5 & 14 & 42 & 132\\
0 & 1 & 2 & 0 & 0 & 0 & 2 & 5 & 14 & 42 & 0 & 0 & 0 & 0 & 0\\
0 & 2 & 5 & 0 & 0 & 0 & 5 & 14 & 42 & 132 & 0 & 0 & 0 & 0 & 0\\
1 & 0 & 0 & 2 & 5 & 14 & 0 & 0 & 0 & 0 & 5 & 14 & 42 & 132 & 429\\
2 & 0 & 0 & 5 & 14 & 42 & 0 & 0 & 0 & 0 & 14 & 42 & 132 & 429 & 1430\\
5 & 0 & 0 & 14 & 42 & 132 & 0 & 0 & 0 & 0 & 42 & 132 & 429 & 1430 & 4862\\
0 & 2 & 5 & 0 & 0 & 0 & 5 & 14 & 42 & 132 & 0 & 0 & 0 & 0 & 0\\
0 & 5 & 14 & 0 & 0 & 0 & 14 & 42 & 132 & 429 & 0 & 0 & 0 & 0 & 0\\
0 & 14 & 42 & 0 & 0 & 0 & 42 & 132 & 429 & 1430 & 0 & 0 & 0 & 0 & 0\\
0 & 42 & 132 & 0 & 0 & 0 & 132 & 429 & 1430 & 4862 & 0 & 0 & 0 & 0 & 0\\
2 & 0 & 0 & 5 & 14 & 42 & 0 & 0 & 0 & 0 & 14 & 42 & 132 & 429 & 1430\\
5 & 0 & 0 & 14 & 42 & 132 & 0 & 0 & 0 & 0 & 42 & 132 & 429 & 1430 & 4862\\
14 & 0 & 0 & 42 & 132 & 429 & 0 & 0 & 0 & 0 & 132 & 429 & 1430 & 4862 & 16796\\
42 & 0 & 0 & 132 & 429 & 1430 & 0 & 0 & 0 & 0 & 429 & 1430 & 4862 & 16796 & s\\
132 & 0 & 0 & 429 & 1430 & 4862 & 0 & 0 & 0 & 0 & 1430 & 4862 & 16796 & s & t
\end{pmatrix}$%
}
\end{equation*}

The matrix \(M_4(\beta)\) satisfies the column relations
\(X^3=Y\), \(X^4=XY\), and \(X^3Y=Y^2\).
Let \(\widehat M_4\) denote the \(12\times12\) compression obtained by
deleting the rows and columns indexed by \(X^3\), \(X^4\), and \(X^3Y\).
Set \(s\equiv\beta_{17}\), \(t\equiv\beta_{08}\), and
    $$\Delta(s,t):=t-s^2+117550s-3454708515.$$
A calculation of the nested determinants gives
\(\det(\widehat M_4)_k=1\) for \(1\leq k\leq11\), and
\(\det\widehat M_4=\Delta(s,t)\). Consequently,
\begin{align}\label{extest}
\begin{split}
&\quad M_4(\beta)\succeq0 \text{ and is $(y-x^3)$-pure}\\
\Longleftrightarrow&\quad
\Delta(s,t)>0\\
\Longleftrightarrow&\quad
t>s^2-117550s+3454708515.
\end{split}
\end{align}
The core matrix \(C[A]\), with rows and columns indexed by
$$
\textit{1},\ X,\ X^{2},\ Y,\ XY,\ X^{2}Y,\ Y^{2},\ XY^{2},\
X^{2}Y^{2},\ Y^{3},\ XY^{3},\ X^{2}Y^{3},\ Y^{4},
$$
is
\begin{equation*}
\resizebox{\linewidth}{!}{%
$\begin{pmatrix}
1 & 0 & 1 & 0 & 2 & 0 & 5 & 0 & 14 & 0 & 42 & 0 & 132\\
0 & 1 & 0 & 2 & 0 & 5 & 0 & 14 & 0 & 42 & 0 & 132 & 0\\
1 & 0 & 2 & 0 & 5 & 0 & 14 & 0 & 42 & 0 & 132 & 0 & 429\\
0 & 2 & 0 & 5 & 0 & 14 & 0 & 42 & 0 & 132 & 0 & 429 & 0\\
2 & 0 & 5 & 0 & 14 & 0 & 42 & 0 & 132 & 0 & 429 & 0 & 1430\\
0 & 5 & 0 & 14 & 0 & 42 & 0 & 132 & 0 & 429 & 0 & 1430 & 0\\
5 & 0 & 14 & 0 & 42 & 0 & 132 & 0 & 429 & 0 & 1430 & 0 & 4862\\
0 & 14 & 0 & 42 & 0 & 132 & 0 & 429 & 0 & 1430 & 0 & 4862 & 0\\
14 & 0 & 42 & 0 & 132 & 0 & 429 & 0 & 1430 & 0 & 4862 & 0 & 16796\\
0 & 42 & 0 & 132 & 0 & 429 & 0 & 1430 & 0 & 4862 & 0 & 16796 & 0\\
42 & 0 & 132 & 0 & 429 & 0 & 1430 & 0 & 4862 & 0 & 16796 & 0 & s\\
0 & 132 & 0 & 429 & 0 & 1430 & 0 & 4862 & 0 & 16796 & 0 & s & A\\
132 & 0 & 429 & 0 & 1430 & 0 & 4862 & 0 & 16796 & 0 & s & A & t
\end{pmatrix}$%
}
\end{equation*}
where \(A\) denotes the auxiliary moment \(\beta_{27}\).

\begin{enumerate}

\item
Let \(s=58786\) and \(t=208012\). Then \(\Delta(s,t)=1\), so by
\eqref{extest}, \(M_4(\beta)\) is positive semidefinite and
\((y-x^3)\)-pure, with 
$\Rank M_4(\beta)=12$. Moreover, \(\det C_k=1\) for \(1\leq k\leq12\), while
\(\det C[A]=1-A^2\). By Sylvester's criterion,
\(C[A]\succ0\) if and only if \(-1<A<1\).
Theorem \ref{Jmeasure} now shows that
 $\beta$ has
representing measures and  that $\mathcal{CV}(L_{\beta})$ is the curve $y = x^{3}$.
Any given point on the curve therefore serves as an atom for some representing measure,
but knowledge of the core variety does not, by itself, give all of the atoms and densities of a measure.

To compute a measure
we will first apply moment matrix extensions. Since we know there is a measure, it is
appropriate to attempt to find a measure in the simplest way, via a flat extension, and we use this
approach. Subsequently, we will sketch an alternate derivation using the
push-forward measure of Remark \ref{scheme}, following the method of the proof of Theorem \ref{Jmeasure}
$(iii) \Longrightarrow (ii) $ and Remark \ref{pf}. 

 We represent a recursively generated moment matrix extension $M_{5}$ of $M_{4}$ in block-matrix form as
\[
M_5=
\begin{pmatrix}
M_4 & B\\
B^t & D
\end{pmatrix}.
\]
The columns of \(B\) are indexed by
\(X^5,X^4Y,X^3Y^2,X^2Y^3,XY^4,Y^5\),
must satisfy $X^{5} = X^2Y$, $X^{4}Y = XY^{2}$
and $X^{3}Y^2 = Y^{3}$
, so moment matrix structure shows that
all of the entries of $B$ are determined except $u \equiv \beta_{27} $, $v \equiv \beta_{18}$,
and $w \equiv \beta_{09}$.
To obtain a flat extension $M_{5}$ we must choose $u$, $v$, and $w$
so that $\operatorname{Ran}B \subseteq \operatorname{Ran}M_{4}$, i.e., $B = M_{4}W$ for some matrix $W$, and so that
$C^{\flat} \equiv B^{t}W$ is a Hankel matrix, which can then serve as block $D$ in a flat extension.
Calculations show that these conditions are equivalent to solving the system
\begin{align*}
0&=C^{\flat}_{44}-C^{\flat}_{53}=u^2-1,\\
0&=C^{\flat}_{63}-C^{\flat}_{54}=11u^2-uv+12,\\
0&=C^{\flat}_{55}-C^{\flat}_{64}
  =77u^2-11uv+v^2-uw-55.
\end{align*}
The two solutions are 
$$(u,v,w)=(1,23,298)
    \quad\text{and}\quad
(u,v,w)=(-1,-23,-298).$$
Choose \((u,v,w)=(1,23,298)\). 
The corresponding flat extension $M_{5}$
     has a unique representing measure, which is 12-atomic, of the form
     $\mu = \sum_{i=1}^{12} \rho_{i}\delta_{(x_{i},y_{i})}$,
where the atoms are the common zeros of the polynomials corresponding to
column dependence relations in $M_{5}$. In this case, it suffices to consider just
 the column relations \(p(X,Y)=0\) and
\(q(X,Y)=0\), where
\begin{align*}
p(x,y)
={}&y^5-\bigl(
77+450x-2450y-1690x^2+5825xy-7175y^2\\
&\qquad
+3528x^2y-1890xy^2+350y^3+3925x^2y^2\\
&\qquad
-976xy^3+90y^4
\bigr),
\end{align*}
and
\begin{align*}
q(x,y)
={}&xy^4-\bigl(
11+65x-364y-237x^2+805xy-980y^2\\
&\qquad
+546x^2y-312xy^2+65y^3+531x^2y^2\\
&\qquad
-131xy^3+12y^4
\bigr).
\end{align*}
There are exactly \(12\) common affine zeros of \(p\), \(q\), and
\(y-x^3\). Since \(\operatorname{rank}M_5=12\), 
these points are the atoms for $\mu$.
(Considering additional column relations, such as the relation for $X^2Y^3$, will not change the set of common zeros.) Approximately, the atoms are
\begin{tiny}
\[
\begin{array}{@{}l@{\qquad}l@{\qquad}l@{}}
(1.98423,7.81224),&
(-1.93717,-7.26944),&
(1.85955,6.43022),\\[2pt]
(-1.75261,-5.38342),&
(1.61803,4.23607),&
(-1.45794,-3.09896),\\[2pt]
(1.27485,2.07193),&
(-1.07165,-1.23073),&
(0.851559,0.617509),\\[2pt]
(-0.618034,-0.236068),&
(0.374763,0.0526343),&
(-0.125581,-0.00198049).
\end{array}
\]
\end{tiny}
To find the densities $\rho_i$ for the measure, let $V$ denote the $12 \times 12$ bivariate 
Vandermonde matrix whose rows are indexed by the monomials in degree-lexicographic order
up to degree 4, but excluding the ``redundant" monomials $x^{3}$, $x^{4}$, $x^{3}y$, and whose entry in
row $x^{i}y^{j}$, column $k$ is $x_{k}^{i}y_{k}^{j}$. Then $V$ is nonsingular, and if
$b$ denotes the column vector of moments $\beta_{ij}$ corresponding to the monomials
indexing the rows of 
$V$, then 
\(\rho:=(\rho_1,\ldots,\rho_{12})^t=V^{-1}b\), where
\[
\begin{aligned}
\rho\approx{}&
(0.00251335,\ 0.00989547,\ 0.0216825,\ 0.0371339,\\
&\quad 0.0552786,\ 0.0749768,
\quad 0.0949905,\ 0.114062,\\ 
&\quad\quad 0.130994,\ 0.144721,\
0.154382,\ 0.159369)^t.
\end{aligned}
\]

We next sketch an alternate approach to computing a representing measure, 
based on the proof of Theorem \ref{Jmeasure} and Remark \ref{pf}. Consider first the case when
$-1 < A < 1$, so that $C[A]$ is positive definite. Here we view the Hankel matrix $H \equiv C[A]$ 
as the univariate moment matrix for the sequence $\gamma$ defined in 
(\ref{univariate-corresponding-sequence}). For each choice
of $\gamma_{2nd+1}\in \mathbb{R}$, there is a unique flat Hankel extension $H^{\flat}[\gamma_{2nd+1}]$
and (following \cite{cf1991}) a $13$-atomic representing measure for $\gamma$, with atoms and
densities computed as in Remark \ref{pf}; this measure then corresponds to a $13$-atomic measure for $\beta$
as shown in the proof of Theorem \ref{Jmeasure} $(iii) \Longrightarrow (ii)$. If instead we use $A=\pm 1$,
then $C[A]$ is singular, but is still positive and recursively generated, with $\Rank H = 12$.
In this case  $H$ is already a flat Hankel extension of $C_{12}$. For $A=1$, the  generating function
of $H$ is 
\[
\begin{aligned}
p(x)={}&1+6x-21x^{2}-35x^{3}+70x^{4}+56x^{5}-84x^{6}\\
       &-36x^{7}+45x^{8}+10x^{9}-11x^{10}-x^{11}+x^{12}.
\end{aligned}
\] 
The $12$ distinct roots of $p$, which are
the support points for a $12$-atomic measure $\nu$ for $H$  are the $x$-coordinates
of the support points for $\mu$ that we found earlier, with the same densities as for $\mu$, so
$\mu$ is the push-forward of $\nu$.

\item Consider next $s=58784$, $t=207972$. Condition
\eqref{extest} is satisfied and nested determinants show that
$\widehat{C}
 \succ 0$. In particular, $C_{11}\succ 0$, but we have
$\det C_{12} = -1$, so for no value of $A$ will $C[A]$ be positive
semidefinite. By Theorem \ref{Jmeasure}, $\beta$
has no measure.

\item Now let
$s= 58785$, $t>207990$. Then $\eqref{extest}$ holds,
  and we have
 $C_{11}\succ 0$;
 however,
 $\det
C_{12}=
0$, so there exists $x \in \mathbb{R}^{12}$ such that
$C_{12}x = 0$. Now $\widehat{r}:= (x^{t},0)\equiv
  (r_{0},\ldots,r_{11},0)$
satisfies  $\langle C\widehat{r}, \widehat{r} \rangle =0$ and, with $\widehat{s}\equiv 0$,
also satisfies the consistency
 requirement
$r_{11}r_{12} +s_{11}s_{12} =0$ (cf.\ \eqref{d3requirement}).
Remark \ref{reverse} now implies that there exists $p\in \ker L_{\beta}$ such that $Q(x):= p(x,x^{3}) =
r(x)^{2}$.
Therefore,
 $\Card \mathcal{CV}(L)\le \deg r
\le 11 < 12 = \Rank
M_{4}$, so $\mathcal{CV}(L_{\beta}) = \emptyset$ by Corollary \ref{coretest},
and thus $\beta$ has no measure.
\end{enumerate}
Thus the example exhibits all three possibilities relevant here:
a positive definite completion of the core matrix, failure of
positive semidefinite completion, and a boundary case in which a
nonnegative polynomial in $\ker L_\beta$ forces the core variety to
be empty.
\end{example}

In the next section we will prove
that the method of the preceding  example applies to
any positive semidefinite $M_{n}(\beta)$ that is $(y-x^{3})$-pure.
We may therefore formulate one solution to the $(y-x^3)$-pure truncated moment problem as follows (see Corollary \ref{thm:one-formulation-v2} and its proof).

\begin{theorem}

 Suppose $M_{n}(\beta)$ is
positive semidefinite and
 $(y-x^{3})$-pure. Then $\beta$ has a representing measure if and only if
$\det C_{3n} > 0$, in which case $\mathcal{CV}(L_{\beta})$ is the curve $y=x^{3}$.
\end{theorem}

\begin{ex}
\label{ex:y=x4}
Consider next the sequence $\beta^{(8)}$,
with $M_{3}$ given by
\begin{tiny}
\begin{equation*}
     \left( \begin{array}{cccccccccc}
 1& 0& 2& 1& 0 &14& 0& 5& 0& 132\\
 0& 1&0& 0 &5 &0 &2& 0& 42& 0\\
 2 &0& 14& 5& 0& 132& 0& 42& 0&  1430\\
 1& 0& 5& 2& 0& 42& 0& 14& 0&  429\\
 0& 5 &0 &0& 42& 0& 14 &0& 429 &0\\
 14& 0 &132& 42& 0 &1430&0 &429& 0 &16796\\
 0& 2& 0& 0& 14& 0& 5&  0& 132& 0\\
 5& 0&  42 & 14& 0& 429& 0  &132& 0& 4862\\
0 & 42 & 0 & 0 & 429 & 0 & 132 &  0 & 4862 & 0\\
132& 0& 1430& 429& 0& 16796& 0 &4862 &0 &208012
\end{array}
\right)
\end{equation*}
\end{tiny}
and the degree 7 and degree 8 blocks given by

\begin{tiny}
\begin{equation*}
     \left( \begin{array}{ccccc}
0& 42& 0& 1430 &0\\
 42& 0 & 1430& 0 & 58786\\
 0 &1430& 0& 58786& 0\\
 1430 &0 & 58786& 0 & 2674440
\end{array}
\right)
\end{equation*}
\end{tiny}

\begin{tiny}
\begin{equation*}
     \left( \begin{array}{ccccc}
14 &0 &429& 0 &16796\\
 0 &429& 0 &16796& 0\\
 429 &0 &16796 &0 & 742900\\
 0& 16796 &0& 742900& 0\\
  16796 &0 & 742900& 0 &353576708
\end{array}
\right)
\end{equation*}
\end{tiny}

The core matrix is a Hankel matrix (see Example \ref{ex:core-4})
    with anti-diagonals completely determined in the first row by
\begin{tiny}
\begin{align*}
\beta_{00} &= 1 ,
    & \beta_{01}&=2,
        &\beta_{02} &= 14,
           &\beta_{03} &=132, \\
\beta_{10} &= 0,
    & \beta_{11}&= 0,
        &\beta_{12} &=0,
            &\beta_{13} &= 0, \\
\beta_{20} &= 1,
    & \beta_{21}&= 5,
        &\beta_{22} &= 42,
            &\beta_{23} &=429, \\
\beta_{30} &=0,
    & \beta_{31}&= 0,
        &\beta_{32} &=0,
            &\beta_{33} &= 0,
\end{align*}
\end{tiny}
and the last column by
\begin{tiny}
\begin{align*}
\beta_{04} &= 1430,
    & \beta_{05}&=16796
        &\beta_{06} &= 208012,
        &\beta_{07}&=2674440,\\
\beta_{14} &= 0,
    & \beta_{15}&= 0,
        &\beta_{16} &=0,
        &\beta_{17}&=0,\\
\beta_{24} &= 4862,
    & \beta_{25}&= 58786,
        &\beta_{26} &= 742900,
        &\beta_{27}&=A_{27},
        \\
\beta_{34} &=0,
    & \beta_{35}&= 0,
        &\beta_{36} &=A_{36},
        &\beta_{37} &=A_{37},\\
&&&&&&\beta_{08} &=353576708.
\end{align*}
\end{tiny}

It is straightforward to verify that
$M_{4}$ is positive semidefinite and $(y-x^{4})$--pure.
Using nested determinants, it is easy to show that $C_{14}\succ 0$. A further calculation
shows that $C_{15}\succ 0$ if and only if $-1 < A_{36}<1$. Setting $A_{36}=0$, we see that
$C_{16}\succ 0$ if and only if $A_{27} = 9694844+f$ for $f>0$. Now
$\det C
= f(318219068-28f-f^{2})-A_{37}^{2}$, so there exists $A_{37}$ such that $C[A]\succ 0$ if and only
if
$0 < f < 96\sqrt{34529}-14~(\approx 17824.7)$. In this case, since $C[A]\succ 0$, the core variety coincides with
the curve $y=x^{4}$.

Fix \(f\) such that
$
0<f<96\sqrt{34529}-14,
$
and set \(A_{36}=0\) and \(A_{27}=9694844+f\). Then
\(C_{16}\succ0\), and there exist values of \(A_{37}\) for which
\(C[A]\succ0\). More precisely, 
the completion is positive definite precisely when
\[
-\sqrt{f(318219068-28f-f^2)}
<A_{37}<
\sqrt{f(318219068-28f-f^2)}.
\]
At either boundary value
$
A_{37}
=\pm\sqrt{f(318219068-28f-f^2)},
$
the matrix \(C[A]\) is positive semidefinite and
\[
\Rank C[A]=16=nd.
\]
Since \(C_{16}\succ0\), each such \(C[A]\) is a flat Hankel
extension of \(C_{16}\) and is therefore recursively generated.
Consequently, \cite{cf1991} yields a \(16\)-atomic representing
measure for the associated univariate moment sequence, whose
push-forward under \(t\mapsto(t,t^4)\) is a \(16\)-atomic
representing measure for \(\beta\).

An even smaller representing measure can be obtained by a more
careful choice of the auxiliary moments \(A_{27}\) and \(A_{36}\),
in addition to \(A_{37}\). The details will appear in Part~2
\cite{fz}, where we prove that, for \(d\geq3\), whenever a
representing measure exists, one can be chosen with at most
\[
\Rank M_n+
\begin{cases}
\dfrac{d^2}{4}-\dfrac{d}{2}-1,
   & \text{if \(d\) is even},\\[0.7em]
\dfrac{d^2}{4}-\dfrac{d}{2}-\dfrac{3}{4},
   & \text{if \(d\) is odd},
\end{cases}
\]
atoms. In the present example, \(\Rank M_4=14\) and \(d=4\), so
this bound is \(15\).
As shown in \cite{fz}, no \(14\)-atomic representing
measure exists in this example, since the only candidate values
\(A_{36}=\pm1\) admit no positive semidefinite, recursively
generated completion \(C[A]\).
\end{ex}

\begin{ex}

Consider next the sequence $\beta^{(8)}$,
defined as in Example \ref{ex:y=x4}, except for the following
5 differences:
\begin{align*}
\beta_{25}&=0,
&\beta_{06}&=3454708516
&\beta_{26}&=3448894372,
&\beta_{07}&=0,
\end{align*}
$$\beta_{08}=2640503382173370698906776695725.$$

It is straightforward to verify that
$M_{4}$ is positive semidefinite and $(y-x^{4})$--pure.
Moreover, $C[A]$ can never be positive semidefinite, since $\beta_{25}=0$ is its 12th diagonal element, but there are nonzero entries in the 12th row and column.
By the converse in Theorem \ref{Jmeasure},
$\beta^{(8)}$ does not admit a representing measure.
\end{ex}

\setcounter{equation}{1}
\section{The  $(y-x^{3}\label{Section 4})$-pure truncated moment problem.}

In this section we apply the previous results to the moment problem for $\beta \equiv \beta^{(2n)}$
where $M_{n}$ is positive semidefinite and  $(y-x^{3})$-pure.
In particular, Theorem \ref{y=x3}  provides a negative answer to Question
\ref{q1}
for $d = 3$.
Let $\Gamma$ stand for the curve $y=x^3$.

 Note that in the core matrix $C$, since $Y=X^{d}$ with $d=3$, there is exactly $1$ auxiliary moment, namely $\beta_{2,2n-1}$, which
we denote by $A \equiv A_{2,2n-1}$ (cf.\ Example \ref{ex-n-d-3}).
Let $\widehat{C}$ be the principal submatrix of $C$ obtained by deleting
row and column $nd$.
Recall from Remark \ref{rem:ort-equiv} that $\widehat C\succ 0$.
 Let $H\equiv H[A]$ denote the matrix obtained from $C\equiv C[A]$
by interchanging rows and columns $nd$ and $nd + 1$ (the last $2$ rows and columns), so that $H$ is orthogonally equivalent to $C$, i.e.,
\begin{equation}
\label{H-permutation}
    H = P^TCP,
\end{equation}
where $P$ is a permutation matrix defined on the standard orthonormal basis $e_1,\ldots,e_{nd+1}$ for $\RR^{nd+1}$ by
$$Pe_i=\left\{
\begin{array}{rl}
e_i,&   i\leq nd-1,\\
e_{nd+1},&  i=nd,\\
e_{nd},& i=nd+1.
\end{array}
\right.$$

We may thus represent $H$ as
\begin{equation}
\label{Hmatrix}
H=      \left( \begin{array}{cc}
\widehat{C} & v\\
v^{t} & \beta_{1,2n-1}
  \end{array}\right),
\end{equation}
with $\widehat{C}\succ 0$ and where $v$ is of the form
\begin{equation}
v=      \left( \begin{array}{c}
h \\
A
  \end{array}\right).
\end{equation}
(Here $h\in \mathbb{R}^{dn-1}$ and $v^{t}$ denotes the row vector transpose of $v$.)
{As in Section \ref{Section 2},
for $1\le j\le dn+1$, let $C_{j}$ denote the compression of $C$ to the first $j$
rows and columns.}
Write
\begin{equation}
\label{def:widehat-C}
\widehat{C}=
\begin{pmatrix}
C_{dn-1} & z\\
z^t & \beta_{0,2n}
\end{pmatrix},
\end{equation}
where
$z\in \RR^{dn-1}$
is of the form
$$z = \begin{pmatrix}
    k\\\beta_{1,2n-1}
\end{pmatrix}$$
for some $k\in \RR^{dn-2}$.
We now have
\begin{equation}\label{H[A]}
H[A]= \left( \begin{array}{ccc}
C_{dn-1} & z & h\\
z^{t} & \beta_{0,2n} & A \\
h^{t} & A & \beta_{1,2n-1}
  \end{array}\right).
\end{equation}

Since $\widehat{C}\succ 0$, $\widehat{C}^{-1}$
exists and has the form
\begin{equation}
\widehat{C}^{-1}=
\left( \begin{array}{cc}
\mathcal{C} & w\\
w^{t} & \epsilon
  \end{array}\right),
\end{equation}
where (see e.g., \cite[p.\ 3144]{f1})
\begin{align}
\label{expr:inverse}
\begin{split}
\epsilon
&=\frac{1}{\beta_{0,2n}-z^tC_{dn-1}^{-1}z}>0,\\
w
&=-\epsilon C_{dn-1}^{-1}z\in \mathbb{R}^{dn-1},\\
\mathcal{C}
&=C_{dn-1}^{-1}(1+\epsilon zz^tC_{dn-1}^{-1})
\in \RR^{(dn-1)\times (dn-1)}.
\end{split}
\end{align}
Now
$$\widehat{C}^{-1}v = \left( \begin{array}{c}
 \mathcal{C}h + Aw \\
w^{t}h +A\epsilon
  \end{array}\right),$$
and we set
\begin{equation}\label{A}A\equiv A_{0}:= -\frac{w^{t}h}{\epsilon},
\end{equation}
 so that
\begin{equation}\label{kervector}
\widehat{C}^{-1}v = \left( \begin{array}{c}
 \mathcal{C}h -\frac{w^th}{\epsilon}w \\
0
  \end{array}\right).
\end{equation}
With this value of $A$ in $C$, and thus also in $v$,
let
\begin{align}\label{rho}
\begin{split}
\phi
&:= v^{t}\widehat{C}^{-1}v
= h^{t}\mathcal{C}h     -\frac{w^{t}h  h^{t}w}{\epsilon}\\
&=(h^tC_{dn-1}^{-1}h
+\epsilon h^tC_1^{-1}zz^tC_{dn-1}^{-1}h)-
\epsilon z^tC_{dn-1}^{-1}hh^tC_{dn-1}^{-1}z\\
&=h^tC_{dn-1}^{-1}h,
\end{split}
\end{align}
where we used \eqref{expr:inverse} in the second equality.

To emphasize the dependence of $\phi$ on $\beta$,  we sometimes denote
$\phi$ as $\phi[\beta]$. In Example \ref{poslinfuncl} (below) we will use the fact that $\phi$
is independent of $\beta_{1,2n-1}$ and $\beta_{0,2n}$. To see this,
note that $\beta_{1,2n-1}$ is an element of vectors $z$  and $z^{t}$, so (\ref{H[A]}) shows
that $C_{dn-1}$ and $h$ are independent of $\beta_{1,2n-1}$ and $\beta_{0,2n}$. It now follows from
\eqref{rho} that $\phi$
is independent of $\beta_{1,2n-1}$ and $\beta_{0,2n}$ as well. Thus, if $\widetilde{\beta}^{(2n)}$ has the property
that $M_{n}(\widetilde{\beta})$ is positive semidefinite and $(y-x^{3})$-pure, and if
$\beta_{ij} = \widetilde{\beta}_{ij}$ for all $(i,j) \not =(1,2n-1)$ and
$(i,j)\neq (0,2n)$, then
$\phi[\widetilde{\beta}] = \phi[\beta]$.
Note that $\phi$ would depend on $\beta_{1,2n-1}$ and $\beta_{0,2n}$ if $A_0$
in \eqref{A} was chosen differently. This is due to the fact that the last row of $\widehat{C}^{-1}v$ in \eqref{kervector} would be non-zero.

\begin{theorem}\label{y=x3}
Suppose $M_{n}$ is positive semidefinite and $(y-x^{3})$-pure.
$\beta\equiv \beta^{(2n)}$ has a representing measure if and only if $\beta_{1,2n-1}> \phi$ (equivalently, $C[A_0]\succ 0$).
In this case, $\mathcal{CV}(L_{\beta}) = \Gamma$, which coincides with the union of supports of all
representing measures (respectively, all finitely atomic representing measures).
\end{theorem}
\begin{proof}
Recall from Remark \ref{rem:ort-equiv} that $\widehat{C}$ is positive definite.
Consider first the case $\beta_{1,2n-1}> \phi$.
It follows from (\ref{Hmatrix}) and \cite[Theorem 1]{A} that
$H$ is positive definite. Since $C$ is orthogonally equivalent to $H$, we see that $C$ is
positive definite, so the existence of representing measures and the
conclusion concerning supports follow from Theorem \ref{Jmeasure}.

We next consider the case when $\beta_{1,2n-1}= \phi$, so that
by \cite[Theorem 1]{A},
$H$ is positive semidefinite, but singular.
Since $\widehat{C} \succ 0$, it follows from (\ref{Hmatrix})
and \eqref{kervector}
that $\ker H$ contains the vector
\begin{equation}\label{pvector}
\widehat{u}:= \left( \begin{array}{c}
 \widehat{C}^{-1}v \\
-1
  \end{array}\right)
\equiv
 \left( \begin{array}{c}
 \mathcal{C}h -\frac{w^th}{\epsilon}w \\
0 \\
-1
  \end{array}\right)
\equiv (r_{0},~r_{1},\ldots, ~r_{dn-2}, ~u_{dn-1},~u_{dn})^{t},
\end{equation}
where $u_{dn-1}=0$ and $u_{dn} =-1$.
From the orthogonal equivalence between $H$ and $C$, based on the interchange of rows and columns $nd$ and $nd+1$,
 it follows that $C$ is positive semidefinite
and that $\ker C$ contains the vector
\begin{equation}
\label{def-r}
    \widehat{r} =  (r_{0},~r_{1},\ldots, ~r_{dn-2},~r_{dn-1} ,~r_{dn})^{t},
\end{equation}
where
\begin{equation}
\label{def-r-v2}
r_{dn-1}=u_{dn}=-1\quad\text{and}\quad r_{dn}=u_{dn-1}=0.
\end{equation}
Let $\widehat{s}\equiv (s_{0},\ldots,~s_{dn})^{t} $
denote the $0$ vector, so that $\langle C\widehat{r},\widehat{r}\rangle +\langle C\widehat{s},\widehat{s}\rangle =0$
and
the auxiliary requirement of \eqref{eq:cons-terms}, $r_{dn-1}r_{dn}+s_{dn-1}s_{dn}=0$, is satisfied. Now, following Remark \ref{reverse},
define $a_{ij} = h_{ij}(\widehat{r},\widehat{s})$ ($0\le i\le 2$, $j\ge 0$, $0<i+j\le 2n$).
Then $p:= \sum a_{ij}f_{ij}$ is an element of $\ker L_{\beta}$
which satisfies $Q(x):= p(x,x^{3}) = R(x)^{2}$, where
    $$R(x) := r_{0} + r_{1}x +\cdots + r_{dn-1}x^{dn-1} + r_{dn}x^{dn}.$$
Since $r_{dn} =0$, $R(x)$ has at most $dn-1$ real zeros, so $p$ has at most $dn-1$ zeros in the curve $y=x^{3}$.
Now $p\in \ker L_{\beta}$ satisfies $p|\Gamma\ge 0$ and 
$$\Card \mathcal{Z}(p|\Gamma) \le dn-1<  \frac{d(2n-d+3)}{2} =\Rank M_{n}$$
(since $d=3$),
so Corollary \ref{coretest}
implies that $\beta$ has no representing measure.

To complete the proof, we consider the case when $\beta_{1,2n-1}< \phi$.
From (\ref{Hmatrix}) and (\ref{pvector}) we have
\begin{align*}
\langle H\widehat{u},~\widehat{u}\rangle
&=
 \langle  \left( \begin{array}{c}
 0_{dn\times 1} \\
v^{t}\widehat{C}^{-1}v- \beta_{1,2n-1}
  \end{array}\right),
\left( \begin{array}{c}
 *_{dn\times 1} \\
-1
  \end{array}\right)
\rangle \\
&= \beta_{1,2n-1} - v^{t}\widehat{C}^{-1}v\\
&= \beta_{1,2n-1}-\phi<0.
\end{align*}

Recall that $H = P^TCP$ (cf.\ \eqref{H-permutation}).
Setting $\widehat{r}:=   P\widehat{u}$, we have
$$\langle C\widehat{r},~\widehat{r} \rangle
 = \langle H\widehat{u},\widehat{u}\rangle < 0,$$
where $\widehat r$ is as in \eqref{def-r}, \eqref{def-r-v2}.

Let $\epsilon
= (\phi -\beta_{1,2n-1})^{1/2}$.
Since $\langle \widehat{C}e_{1},~e_{1}\rangle = \beta_{00} =1$,  then   the constant polynomial
$S(x) = \epsilon$, with coefficient vector $\widehat{s} = (\epsilon, 0,\ldots, 0)^{t}$,
satisfies $s_{dn-1}s_{dn} =0$, and we have
$ \langle C \widehat{r},~\widehat{r} \rangle + \langle C\widehat{s},~\widehat{s} \rangle = 0$.
So $\widehat{r}$ and $\widehat{s}$ together satisfy the
  auxiliary requirement
of \eqref{eq:cons-terms}.
Constructing $p(x,y)$ as in Remark \ref{reverse}, we have that $p\in \ker L_{\beta}$.
Now, $p(x,x^{d}) = R(x)^{2}+ S(x)^{2}   \ge \epsilon^{2} >0$.
Since $p$ is strictly
positive on $\Gamma$, then $\mathcal{CV}(L_{\beta}) =\emptyset$, and therefore $\beta$ has no representing measure.
\end{proof}
\begin{rem}
 In Theorem \ref{y=x3}, an alternative proof of the case $\beta_{1,2n-1}<\phi$
can be based on Theorem \ref{Jmeasure},  as follows. Let $A_{0}$ be as in (\ref{A}).
If $\beta_{1,2n-1}<\phi[A_{0}]$, then (\ref{rho}) implies that
$\beta_{1,2n-1}<
 h^{t}C_{dn-1}^{-1}h$.
It therefore follows from (\ref{H[A]}) that for every $A\in \mathbb{R}$, the matrix
$\left( \begin{array}{cc}
C_{dn-1} & h\\
h^{t} & \beta_{1,2n-1}
  \end{array}\right)$
is a principal submatrix of $H[A]$ that is not positive semidefinite.
Thus, for every $A$, $H[A]$, and hence $C[A]$, is not positive semidefinite, so Theorem \ref{Jmeasure}
implies that $\beta$ has no representing measure.
\end{rem}

\begin{cor}
\label{thm:one-formulation-v2}
     Suppose $M_{n}(\beta)$ is
positive semidefinite and
 $(y-x^{3})$-pure. Then $\beta$ has a representing measure if and only if
$\det C_{dn} > 0$, in which case $\mathcal{CV}(L_{\beta})$ is the curve $y=x^{3}$.
\end{cor}

\begin{proof}
    Note that $C_{dn}$ is equal to
        $
        \begin{pmatrix}
            C_{dn-1} & h \\ h^t & \beta_{1,2n-1}
        \end{pmatrix}
        $ (cf.\ \eqref{H[A]}).
    From $\widehat{C}\succ 0$ it follows that $C_{dn-1}\succ 0$ (cf.\ \eqref{def:widehat-C}).
    Using \cite{A}, we have that
    $$
        C_{dn}\succ 0
        \;\;\Longleftrightarrow\;\; \beta_{1,2n-1}>h^tC_{dn-1}^{-1}h
        \;\;\underbrace{\Longleftrightarrow}_{\eqref{rho}}\;\; \beta_{1,2n-1}>\phi.
    $$
    Now the statement of the corollary follows from Theorem \ref{y=x3}.
\end{proof}

In \cite{f1} a rather lengthy  construction with moment matrices is used to derive
a certain rational expression in the moment data, denoted by $\psi$ in \cite{f1},
such that $\beta$ has a representing measure if and only if $\beta_{1,2n-1} > \psi$,
in which case $M_{n}$ admits a flat extension $M_{n+1}$. In view of Theorem \ref{y=x3},
it is clear that $\psi = \phi$ (although this is not at all apparent from the definitions of
these expressions).

\begin{cor}
Suppose $M_{n}(\beta)$ is positive semidefinite and $(y-x^{3})$-pure. The following are equivalent:
\begin{enumerate}
    \item  $\beta$ has a representing measure;
    \smallskip
    \item $\beta$ has a finitely atomic measure;
    \smallskip
    \item $M_{n}(\beta)$ has a flat extension $M_{n+1}$;
    \smallskip
    \item $\mathcal{CV}(L_{\beta}) \not = \emptyset$;
    \smallskip
    \item With $A$ and $\phi$ defined by \eqref{A} and \eqref{rho}, respectively,
    \[
    \beta_{1,2n-1}>\phi;
    \]
    \item $\mathcal{CV}(L_{\beta}) = \Gamma$.
\end{enumerate}
\end{cor}
\begin{proof}
The implications $(i)  \Longrightarrow (iv) \Longrightarrow (ii) \Longrightarrow (i)$ follow from
the Core Variety Theorem and its proof. The equivalence of $(i)$ and $(iii)$ is established in \cite{f1},
and the equivalence of $(i)$, $(v)$, and $(vi)$
is Theorem \ref{y=x3}.
\end{proof}

In \cite{ef} the authors used the results of \cite{f1} to exhibit a family of positive  ($y-x^{3}$)-pure moment matrices
$M_{3}(\beta^{(6)})$
such that $\beta^{(6)}$ has no representing measure but the Riesz functional is positive (cf. Section \ref{sec:prel}).
Here, positivity of the functional cannot be derived from positivity of $M_{3}$ using an argument such as
$L(p) = L(\sum p_{i}^{2}) = \sum \langle M_{3}\widehat{p_{i}},\widehat{p_{i}}\rangle \ge 0$,
because, by the
theorem of Hilbert, not every nonnegative polynomial $p(x,y)$  of degree 6 can
be represented as a sum of squares. Using Theorem \ref{y=x3} we can extend this example
to a family of ($y-x^{3}$)-pure matrices $M_{n}$, for $n\ge 3$ as follows.
\begin{example}\label{poslinfuncl}
Suppose $M \equiv M_{n}(\beta)$ is positive semidefinite and ($y-x^{3}$)-pure. Let $\phi\equiv \phi[\beta]$ be as in (\ref{rho})
and
suppose $\phi= \beta_{1,2n-1}$, so that $\beta$ has no representing measure by Theorem \ref{y=x3}.
We claim that the Riesz functional $L_{\beta}$ is positive.
Let $\widehat{M}$ denote the central compression  of $M$ to rows and columns that are of the form
$X^{i}Y^{j}$ with $0\le i < 3$, so that $\Rank M = \Rank\widehat{M}$ and $\widehat{M}\succ 0$.
Now let $\widetilde{\beta}$  be defined to coincide with $\beta$, except possibly in the $\beta_{1,2n-1}$
position.
It follows from the structure of positive matrices that there exists $\delta >0$ such that if
$|\widetilde{\beta}_{1,2n-1} - \beta_{1,2n-1}|< \delta$, then $\widehat{M_{n}}(\widetilde{\beta})$
is positive definite. The structure of positive ($y-x^{3}$)-pure moment matrices now implies that
$M_{n}(\widetilde{\beta})$ is positive semidefinite
and ($y-x^{3}$)-pure.  Now consider
the sequence $\beta^{[m]}$ which coincides with $\beta$ except that $\beta^{[m]}_{1,2n-1} = \beta_{1,2n-1}+1/m$.
It follows that there exists $m_{0} > 0$ such that if $m>m_{0}$, then $M^{[m]} \equiv M_{n}(\beta^{[m]})$
is positive semidefinite and ($y-x^{3}$)-pure. By the remarks preceding Theorem \ref{y=x3}, we have
$\beta^{[m]}_{1,2n-1} = \beta_{1,2n-1}+1/m > \beta_{1,2n-1} =\phi[M] = \phi[M_{n}(\beta^{[m]})]$,
so Theorem \ref{y=x3} implies that $\beta^{[m]}$ has a representing measure. Thus, $L_{\beta^{[m]}}$ is positive,
and since the cone of sequences with positive functionals is closed, it follows that $L_{\beta}$ is positive.
\end{example}
To exhibit $M_{n}(\beta)$ as in Example \ref{poslinfuncl}, we may start with any positive semidefinite $(y-x^{3})$-pure
$M_{n}(\beta')$. Define $\beta$ so that it coincides with $\beta'$ except that $\beta_{1,2n-1}=\phi[\beta']$.
If necessary, increase $\beta_{0,2n}$ to ensure positivity of $M_{n}(\beta)$.  Then $M_{n}(\beta)$ is positive
semidefinite, $(y-x^{3})$-pure, and $\beta_{1,2n-1} = \phi[\beta'] = \phi[\beta]$  by the remarks preceding
Theorem \ref{y=x3}.

\section*{Declarations}

\noindent \textbf{Competing interests}
The authors declare that they have no known competing financial interests or personal relationships that could have appeared to influence the work reported in this paper.

\noindent \textbf{Author Contribution} The authors have equal contributions to each part of this article. All the authors read
and approved the final manuscript.

\noindent \textbf{Funding Information}
Alja\v{z} Zalar was supported by ARIS research core funding No.~P1-0288 and grants No.~J1-50002, J1-60011, and J1-70017.

\noindent \textbf{Data availability Statement}
No data were used for the research described in this article.

\noindent{\bf{Acknowledgement.}} The first-named author is grateful to Ra\'ul Curto for helpful discussions
concerning core varieties for TKMP for certain quadratic planar curves during a visit to the University of Iowa
in Fall, 2019.

\end{document}